\documentclass[11pt]{amsart}
\usepackage{amsmath,amsthm,amsfonts,amssymb,times}
\usepackage{latexsym}

\usepackage{mathrsfs} 

\numberwithin{equation}{section}  

\DeclareMathAlphabet{\curly}{U}{rsfs}{m}{n}  

\theoremstyle{remark}
\newtheorem{remark}{Remark}
\theoremstyle{plain}

\newtheorem{lem}{Lemma}[section]
\newtheorem{thm}{Theorem}
\newtheorem{cor}{Corollary}
\newtheorem{definition}{Definition}

\newcommand{\Z}{\mathbb{Z}}
\newcommand{\R}{\mathbb{R}}
\newcommand{\E}{\mathbb{E}}   
\newcommand{\PR}{\mathbb{P}}  

\makeatletter
\renewcommand{\pmod}[1]{\allowbreak\mkern7mu({\operator@font mod}\,\,#1)}
\makeatother

\newcommand{\bal}{\[\begin{aligned}}
\newcommand{\eal}{\end{aligned}\]}

\newcommand{\be}{\begin{equation}}
\newcommand{\ee}{\end{equation}}

\newcommand{\ssum}[1]{\sum_{\substack{#1}}}  

\newcommand{\eps}{\ensuremath{\varepsilon}}

\renewcommand{\le}{\leqslant}
\renewcommand{\leq}{\leqslant}
\renewcommand{\ge}{\geqslant}
\renewcommand{\geq}{\geqslant}

\newcommand{\order}{\asymp}      
\renewcommand{\(}{\left(}
\renewcommand{\)}{\right)}

\newcommand{\pfrac}[2]{\left(\frac{#1}{#2}\right)}  

\newcommand{\ba}{\ensuremath{\mathbf{a}}}

\newcommand{\bfe}{\ensuremath{\mathbf{e}}}
\newcommand{\bfW}{\ensuremath{\mathbf{W}}}

\newcommand{\asym}{\sim}   

\newcommand{\PP}{\mathcal{P}}
\newcommand{\QQ}{\mathcal{Q}}
\newcommand{\cS}{\mathcal{S}}
\newcommand{\cR}{\mathcal{R}}

\renewcommand{\mod}{\bmod}  

\begin{document}

\title{Long gaps between primes}

\author{Kevin Ford}
\address{Department of Mathematics\\ 1409 West Green Street \\ University
of Illinois at Urbana-Champaign\\ Urbana, IL 61801\\ USA}
\email{ford@math.uiuc.edu}

\author{Ben Green}
\address{Mathematical Institute\\
Radcliffe Observatory Quarter\\
Woodstock Road\\
Oxford OX2 6GG\\
England }
\email{ben.green@maths.ox.ac.uk}

\author{Sergei Konyagin}
\address{Steklov Mathematical Institute\\
8 Gubkin Street\\
Moscow, 119991\\
Russia}
\email{konyagin@mi.ras.ru}

\author{James Maynard}
\address{Mathematical Institute\\
Radcliffe Observatory Quarter\\
Woodstock Road\\
Oxford OX2 6GG\\
England }
\email{james.alexander.maynard@gmail.com}

\author{Terence Tao}
\address{Department of Mathematics, UCLA\\
405 Hilgard Ave\\
Los Angeles CA 90095\\
USA}
\email{tao@math.ucla.edu}

\begin{abstract} Let $p_n$ denote the $n$-th prime.  We prove that
\[\max_{p_{n+1} \leq X} (p_{n+1}-p_n)  \gg \frac{\log X \log \log X\log\log\log\log X}{\log \log \log X}\]
for sufficiently large $X$, improving upon recent bounds 
of the first three and fifth authors and of the fourth author.
 Our main new ingredient is a generalization of a hypergraph covering
 theorem of Pippenger and Spencer, proven using the R\"odl nibble method.
\end{abstract}

\maketitle

\setcounter{tocdepth}{1}
\tableofcontents

\section{Introduction}

Let $p_n$ denote the $n^{\operatorname{th}}$ prime, and let
$$
G(X) := \max_{p_{n+1} \leq X} (p_{n+1}-p_n)
$$
denote the the maximum gap between consecutive primes less than $X$. It is clear from the prime number theorem that
\[ G(X) \geq (1 + o(1)) \log X,\]
as the \emph{average} gap between the prime numbers which are $\le X$ is $\asym \log X$. 
In 1931, Westzynthius \cite{West} proved that infinitely often, the gap between consecutive 
prime numbers can be an arbitrarily large multiple of the average gap, that is,
$G(X)/\log X\to \infty$ as $X\to\infty$, improving upon prior results of Backlund \cite{back} and Brauer-Zeitz \cite{brauer}.  Moreover, he proved the quantitative
bound\footnote{As usual in the subject, $\log_2 x = \log \log x$, $\log_3 x = \log \log \log x$, and so on.  The conventions for asymptotic notation such as $\ll$ and $o()$ will be defined in Section \ref{not-sec}.}
\[
G(X) \gg \frac{\log X \log_3 X}{\log_4 X}.
\]
In 1935 Erd\H{o}s \cite{erdos-gaps} sharpened this to
\[ G(X) \gg \frac{\log X \log_2X}{(\log_3X)^2}\] and in 1938 Rankin \cite{R1} 
made a subsequent improvement
\[ G(X) \geq (c + o(1)) \frac{\log X \log_2 X \log_4 X}{(\log_3 X)^2}\] with $c = \frac{1}{3}$. The constant $c$ was increased several times: to $\frac{1}{2}e^{\gamma}$ by Sch\"onhage \cite{schonhage}, then to $c = e^{\gamma}$ by Rankin \cite{rankin-1963}, to $c = 1.31256 e^{\gamma}$ by Maier and Pomerance \cite{MP} and, most recently, to $c = 2e^{\gamma}$ by Pintz \cite{P}.

Recently, in two independent papers \cite{FGKT,maynard-large}, the authors showed that $c$ could be taken to be arbitrarily large, answering in the affirmative a long-standing conjecture of Erd\H os \cite{Erd90}.  The methods of proof in \cite{FGKT} and \cite{maynard-large} differed in some key aspects. The arguments in \cite{FGKT} used recent results \cite{gt-linearprimes,gt-nilmobius, GTZ} on the number of solutions to linear equations in primes, whereas the arguments in \cite{maynard-large} instead relied on multidimensional
prime-detecting
sieves introduced in \cite{maynard-gaps}.  The latter arguments have the advantage of coming with quantitative control on the error terms, as worked out in \cite{maynard-dense}.  Using this, in unpublished work of the fourth author the above bound was improved to
\begin{equation}\label{G1-bound}
 G(X) \gg \frac{\log X \log_2 X}{\log_3 X}
\end{equation}
for sufficiently large $X$. 

Our main theorem is the following further quantitative improvement.

\begin{thm}[Large prime gaps]\label{mainthm} For any sufficiently large $X$, one has
$$ G(X) \gg \frac{\log X \log_2 X \log_4 X}{\log_3 X}.$$
The implied constant is effective.
\end{thm}

Our arguments combine ideas from the previous papers \cite{FGKT,maynard-large}, and also involve a new generalization of a hypergraph covering theorem of Pippenger and Spencer \cite{pippenger} which is of independent interest.  In a sequel \cite{FMT} to this paper, a subset of the authors will extend the above theorem to also cover chains of consecutive large gaps between primes, by combining the methods in this paper with the Maier matrix method.  In view of this, we have written some of the key propositions in this paper in slightly more generality than is strictly necessary to prove Theorem \ref{mainthm}, as the more general versions of these results will be useful in the sequel \cite{FMT}.

The results and methods of this paper have also subsequently been applied by Maier and Rassias \cite{MR} (extending the method of the first author, Heath-Brown and the third author \cite{FHBK}) to obtain large prime gaps (of the order of that in Theorem \ref{mainthm}) that contain a perfect $k^{\operatorname{th}}$ power of a prime for a fixed $k$, and by Baker and Freiberg \cite{BF} to obtain lower bounds on the density of limit points of prime gaps normalized by any function that grows slightly slower than the one in Theorem \ref{mainthm}.  We refer the interested reader to these papers for further details.

\subsection{Historical background}

Based on a probabilistic model of primes, Cram\'er \cite{Cra} conjectured
that 
\[
\limsup_{X\to\infty} \frac{G(X)}{\log^2 X} = 1. 
\]
Granville \cite{Gra} offered a refinement of Cram\'er's model and
has conjectured that the $\limsup$ above is in fact at least $2e^{-\gamma}=1.1229\ldots$.  These conjectures are well beyond the reach of our methods.
Cram\'er's model also predicts that the normalized prime gaps
$\frac{p_{n+1}-p_n}{\log p_n}$ should have exponential distribution, that is,
$p_{n+1}-p_n \ge C\log p_n$ for about $e^{-C}\pi(X)$ primes $\le X$, for any fixed $C>0$.
Numerical evidence from prime calculations up to 
$4\cdot 10^{18}$ \cite{numerical-tos}
matches this prediction quite closely, with the exception of values of $C$
close to $\log X$, for which there is very little data available.  In fact,
$\max_{X\le 4\cdot 10^{18}} G(X)/\log^2 X \approx 0.9206$, slightly below the
predictions of Cram\'er and Granville.

Unconditional upper bounds for $G(X)$ are far from the conjectured
truth, the best being
$G(X) \ll X^{0.525}$ and due to Baker, Harman and Pintz \cite{BHP}.
Even the Riemann Hypothesis only\footnote{Some slight improvements are available if one also assumes some form of the pair correlation conjecture; see \cite{heath}.} furnishes the bound $G(X)\ll X^{1/2}\log X$ \cite{Cra1920}.

All works on lower bounds for $G(X)$ have followed a similar overall plan
of attack: show that there are at least $G(X)$ consecutive
integers in $(X/2,X]$, 
each of which has a ``very small'' prime factor.  
To describe the results, we make the following definition.

\begin{definition}\label{y-def}
Let $x$ be a positive integer. Define $Y(x)$ to be the largest integer
$y$ for which one may select residue classes $a_p \mod p$, one for
each prime $p \leq x$, which together ``sieve out'' (cover) the whole
interval $[y] = \{1,\dots,\lfloor y\rfloor\}$.  Equivalently, $Y(x)$
is the largest integer $m$ so that there are $m$ consecutive integers
coprime to $P(x)$.
\end{definition}

The relation between this function $Y$ and gaps between primes is encoded in the following simple lemma.

\begin{lem}\label{lem11}
Write $P(x)$ for the product of the primes less than or equal to $x$. Then
$$G(P(x)+ Y(x) + x) \geq Y(x).$$
\end{lem}

\begin{proof}
Set $y = Y(x)$, and select residue classes $a_p \mod p$, one for each prime $p \leq x$, which cover $[y]$. By the Chinese remainder theorem there is some $m$, $x <  m \leq x +P(x)$, with $m \equiv -a_p \pmod{p}$ for all primes $p \leq x$. We claim that all of the numbers $m+1,\dots, m+y$ are composite, which means that there is a gap of length $y$ amongst the primes less than $m+y$, thereby concluding the proof of the lemma. To prove the claim, suppose that $1 \leq t \leq y$. Then there is some $p$ such that $t \equiv a_p \pmod{p}$, and hence $m + t \equiv -a_p + a_p \equiv 0 \pmod{p}$, and thus $p$ divides $m + t$. Since $m+t > m > x \geq p$, $m + t$ is indeed composite.
\end{proof}

By the prime number theorem we have $P(x) = e^{(1 + o(1))x}$. It turns out (see below) that $Y(x)$ has size $x^{O(1)}$. Thus the bound of Lemma \ref{lem11} implies that
$$
 G(X) \geq Y\big((1 + o(1)) \log X\big)
$$
as $X \to \infty$.  In particular, Theorem \ref{mainthm} is a consequence of the bound
\begin{equation}\label{Y-lower}
Y(x) \gg \frac{x \log x \log_3 x}{\log_2 x},
\end{equation}
which we will establish later in this paper.  This improves on the bound $Y(x) \gg \frac{x \log x\log_3 x}{\log_2^2 x}$ obtained by Rankin \cite{R1}, and the improvement $Y(x) \gg \frac{x \log x}{\log_2 x}$ obtained in unpublished work of the fourth author.

The function $Y$ is intimately related to \emph{Jacobsthal's function} $j$. If $n$ is a positive integer then $j(n)$ is defined to be the maximal gap between integers coprime to $n$. In particular $j(P(x))$ is the maximal gap between numbers free of prime factors $\leq x$, or equivalently $1$ plus the longest string of consecutive integers, each divisible by some prime $p \leq x$.  The Chinese remainder theorem construction given in the proof of Lemma \ref{lem11} in fact proves that 
\begin{equation}\label{j-y}
 Y(x) = j(P(x))-1. 
\end{equation}
This observation, together with results in the literature, gives upper bounds for $Y$. The best upper bound known is $Y(x) \ll x^2$, which comes from Iwaniec's work \cite{Iw} on Jacobsthal's function. It is conjectured by Maier and Pomerance that in fact $Y(x) \ll x (\log x)^{2 + o(1)}$. This places a serious (albeit conjectural) upper bound on how large gaps between primes we can hope to find via lower bounds for $Y(x)$: a bound in the region of $G(X) \gtrapprox \log X (\log \log X)^{2 + o(1)}$, far from Cram\'er's conjecture, appears to be the absolute limit of such an approach.

The lower bound on certain values of Jacobsthal's function provided by \eqref{Y-lower}, \eqref{j-y} can be inserted directly into \cite[Theorem 1]{Pom}
 to obtain a lower bound for the maximum over $l$ of $p(k,l)$, the least
prime in the arithmetic progression $l\mod k$, in the case when the modulus $k$ has no small prime factors.  We have

\begin{cor}  For any natural number $k$, let $M(k)$ denote the maximum value of $p(k,l)$ over all $l$ coprime to $k$.  Suppose that $k$ has no prime factors less than or equal to $x$ for some $x \leq \log k$.  Then, if $x$ is sufficiently large (in order to make $\log_2 x$, $\log_3 x$ positive), one has the lower bound
$$ M(k) \gg k \frac{x \log x \log_3 x}{\log_2 x}.$$
\end{cor}

\begin{proof} Apply [36, Theorem 1]  with $m=P(x)$ if
$x\le \frac12 \log k$ and with $m=P(\frac12\log k)$ if $\frac12 \log k<x\le\log k$.
\end{proof}

In view of \cite[Theorem 3]{Pom}, one may also expect to be able to prove a lower bound of the form
$$ M(k) \gg \phi(k) \frac{\log k \log_2 k \log_4 k}{\log_3 k}$$
for a set of natural numbers $k$ of density $1$, but we were unable to find a quick way to establish this from the results in this paper.

\subsection{Method of proof}

Our methods here are a combination of those in our previous papers \cite{FGKT,maynard-large}, which are in turn based in part on arguments in previous papers, particularly those of Rankin \cite{R1} and Maier-Pomerance \cite{MP}; we also modify some arguments of Pippenger and Spencer \cite{pippenger} in order to make the lower bound in Theorem \ref{mainthm} as efficient as possible.

As noted above, to prove Theorem \ref{mainthm}, it suffices to sieve out an interval $[y]$ by residue classes $a_p \mod p$ for each prime $p \leq x$, where $y \order \frac{x \log x \log_3 x}{\log_2 x}$.  Actually, it is permissible to have $O(\frac{x}{\log x})$ survivors in $[y]$ that are not sieved out by these residue classes, since one can easily eliminate such survivors by increasing $x$ by a constant multiplicative factor.  Also, for minor technical reasons, it is convenient to sieve out $[y] \backslash [x]$ rather than $[y]$.

Following \cite{FGKT}, we will sieve out $[y] \backslash [x]$ by the residue classes $0 \mod p$ both for very small primes $p$ ($p \leq \log^{20} x$) and medium primes $p$ (between $z := x^{\log_3 x/(4\log_2 x)}$ and $x/2$).  The survivors of this process are essentially the set $\QQ$ of primes between $x$ and $y$.  After this initial sieving, the next stage will be to randomly sieve out residue classes $\mathbf{\vec a} = (\mathbf{a}_s \mod s)_{s \in \cS}$ for small primes $s$ (between $\log^{20} x$ and $z$).  (This approach differs slightly from the approach taken in \cite{maynard-large} and earlier papers, in which the residue classes $1 \mod s$ for small (and very small) primes are used instead.)  This cuts down the set of primes $\QQ$ to a smaller set $\QQ \cap S(\mathbf{\vec a})$, whose cardinality is typically on the order of $\frac{x}{\log x} \log_2 x$.  The remaining task is then to select integers $n_p$ for each prime $p$ between $x/2$ and $x$, such that the residue classes $n_p \mod p$ cut down $\QQ \cap S(\mathbf{\vec a})$ to a set of survivors of size $O( \frac{x}{\log x})$.

Assuming optimistically that one can ensure that the different residue classes $n_p \mod p$ largely remove disjoint sets from $\QQ \cap S(\mathbf{\vec a})$, we are led to the need to select the integers $n_p$ so that each $n_p \mod p$ contains about $\log_2 x$ of the primes in $\QQ \cap S(\mathbf{\vec a})$.  In \cite{FGKT}, the approach taken was to use recent results on linear equations in primes \cite{gt-nilmobius,gt-linearprimes, GTZ} to locate arithmetic progressions $q, q+r! p, \dots, q+(r-1)r!p$ consisting entirely of primes for some suitable $r$, and then to take $n_p = q$.  Unfortunately, due to various sources of ineffectivity in the known results on linear equations in primes, this method only works when $r$ is fixed or growing extremely slowly in $x$, whereas here we would need to take $r$ of the order of $\log_2 x$.  To get around this difficulty, we use instead the methods from \cite{maynard-large}, which are based on the multidimensional 
sieve methods introduced in \cite{maynard-gaps} to obtain bounded intervals with many primes.  A routine modification of these methods gives tuples $q+h_1 p, \dots, q+h_k p$ which contain $\gg \log k$ primes, for suitable large $k$; in fact, by using the calculations in \cite{maynard-dense}, one can take $k$ as large as $\log^c x$ for some small absolute constant $c$ (e.g. $c=1/5$), so that the residue class $q \mod p$ is guaranteed to capture $\gg \log_2 x$ primes in $\QQ$.  

There is however a difficulty due to the overlap between the residue classes $n_p \mod p$.  In both of the previous papers \cite{FGKT,maynard-large}, the residue classes were selected randomly and independently of each other, but this led to a slight inefficiency in the sieving: with each residue class $n_p \mod p$ containing approximately $\log_2 x$ primes, probabilistic heuristics suggest that one would have needed the original survivor set $\QQ \cap S(\mathbf{\vec a})$ to have size about $\frac{x}{\log x} \frac{\log_2 x}{\log_3 x}$ rather than $\frac{x}{\log x} \log_2 x$ if one is to arrive at $O( \frac{x}{\log x} )$ after the final sieving process.  This is what ultimately leads to the additional loss of $\log_4 x$ in \eqref{G1-bound} compared to Theorem \ref{mainthm}.  To avoid this difficulty, we use some ideas from the literature on efficient hypergraph covering.  Of particular relevance is the work of Pippenger and Spencer \cite{pippenger} in which it is shown that whenever one has a large hypergraph $G = (V,E)$ which is uniform both in the sense of edges $e \in E$ having constant cardinality, and also in the sense of the degrees $\# \{ e \in E: v \in e \}$ being close to constant in $v$, one can efficiently cover most of $V$ by almost disjoint edges in $E$.  Unfortunately, the results in \cite{pippenger} are not directly applicable for a number of technical reasons, the most serious of which is that the analogous hypergraph in this case (in which the vertices are the sifted set $\QQ \cap S(\mathbf{\vec a})$ and the edges are sets of the form $\{ q \in \QQ \cap S(\mathbf{\vec a}): q \equiv n_p \pmod p \}$ for various $n_p,p$) does not have edges of constant cardinality.  However, by modifying the ``R\"odl nibble'' or ``semi-random'' method used to prove the Pippenger-Spencer theorem, we are able to obtain a generalization of that theorem in which the edges are permitted to have variable cardinality.  This generalization is purely combinatorial in nature and may be of independent interest beyond the application here to large prime gaps.

We will make a series of reductions to prove Theorem \ref{mainthm}.  To aid the
reader, we summarize the chain of implications below, indicating in which 
Section each implication or Theorem is proven (beneath), and in which Section 
one may find a statement of each Theorem (above).
\[
\overset{\S \ref{sec:weight}}
  {\underset{\S \ref{sec:sievemulti},\ref{sec:verification}}
    {\text{Thm \ref{weight}}}} \,  \underset{\S \ref{sec:weight}}{\implies}
\overset{\S \ref{sec:pip}}{\text{Thm \ref{sieve-primes-2}}} \,  \underset{\S \ref{sec:pip},\ref{pack-proof}}{\implies}
\overset{\S \ref{sec:initial}}{\text{Thm \ref{sieve-primes}}} \,  \underset{\S \ref{sec:initial}}{\implies}
\,\text{Thm \ref{mainthm}}
\]

\subsection{Acknowledgments}

The research of SK was partially performed while he was
visiting KF at the University of Illinois at 
Urbana--Champaign.  Research of KF and SK was
also carried out in part at the University of Chicago.
KF and SK are thankful to Prof. Wilhelm Schlag
for hosting these visits.  KF also thanks the
hospitality of the Institute of Mathematics and Informatics of the
Bulgarian Academy of Sciences.

Also, the research of the SK and TT was partially
performed while SK was 
visiting the Institute for Pure and Applied Mathematics (IPAM) at UCLA,
which is supported by the National Science Foundation.

The research of JM was conducted partly while  he was a CRM-ISM postdoctoral fellow at the Universit\'e de Montr\'eal, and partly while he was a Fellow by Examination at Magdalen College, Oxford.  He also thanks Andrew Granville and Daniel Fiorilli for many useful comments and suggestions.  TT also thanks Noga Alon and Van Vu for help with the references.

KF was supported by NSF grant DMS-1201442.
BG was supported by ERC Starting Grant 279438, \emph{Approximate algebraic structure}. 
TT was supported by a Simons Investigator grant, the
James and Carol Collins Chair, the Mathematical Analysis \&
Application Research Fund Endowment, and by NSF grant DMS-1266164. 

\section{Notational conventions}\label{not-sec}
  
In most of the paper, $x$ will denote an asymptotic parameter going to infinity, with many quantities allowed to depend on $x$.
The symbol $o(1)$ will stand for a quantity tending to zero as $x \to \infty$.
The same convention applies to the asymptotic
notation
$X \asym Y$, which means $X=(1+o(1))Y$, and $X \lesssim Y$, which means $X \leq (1+o(1)) Y$.  We use $X = O(Y)$, $X \ll Y$, and $Y \gg X$ to denote the claim that there is a constant $C>0$ such that $|X| \le CY$ throughout the domain of 
the quantity $X$.
We adopt the convention that $C$ is independent of any parameter
unless such dependence is indicated, e.g. by subscript such as  $\ll_k$. 
 In all of our estimates here, the constant $C$ will be effective (we will not rely on ineffective results such as Siegel's theorem).  If we can take the implied constant $C$ to equal $1$, we write $f = O_{\leq}(g)$ instead.  Thus for instance
$$ X = (1 + O_{\leq}(\eps)) Y$$
is synonymous with
$$ (1-\eps) Y \leq X \leq (1+\eps) Y.$$
Finally, we use $X \order Y$ synonymously with $X \ll Y \ll X$.

When summing or taking products over the symbol $p$, it is understood that $p$ is restricted to be prime.

Given a modulus $q$ and an integer $n$, we use $n \mod q$ to denote the congruence class of $n$ in $\Z/q\Z$.

Given a set $A$, we use $1_A$ to denote its indicator function, thus $1_A(x)$ is equal to $1$ when $x \in A$ and zero otherwise.  Similarly, if $E$ is an event or statement, we use $1_E$ to denote the indicator, equal to $1$ when $E$ is true and $0$ otherwise.  Thus for instance $1_A(x)$ is synonymous with $1_{x \in A}$.


We use $\# A$ to denote the cardinality of $A$, and
for any positive real $z$, we let $[z] := \{ n \in \mathbf{N}: 1 \leq
n \leq z \}$ denote the set of natural numbers up to $z$.

Our arguments will rely heavily on the probabilistic method.  Our random variables will mostly be discrete (in the sense that they take at most countably many values), although we will occasionally use some continuous random variables (e.g. independent real numbers sampled uniformly from the unit interval $[0,1]$).  As such, the usual measure-theoretic caveats such as ``absolutely integrable'', ``measurable'', or ``almost surely'' can be largely ignored by the reader in the discussion below.  We will use boldface symbols such as $\mathbf{X}$ or $\mathbf{a}$ to denote random variables (and non-boldface symbols such as $X$ or $a$ to denote deterministic counterparts of these variables).  Vector-valued random variables will be denoted in arrowed boldface, e.g. $\vec{\mathbf{a}} = (\mathbf{a}_s)_{s \in \cS}$ might denote a random tuple of random variables $\mathbf{a}_s$ indexed by some index set $\cS$.  

We write $\PR$ for probability, and $\E$ for expectation.   If $\mathbf{X}$ takes at most countably many values, we define the \emph{essential range} of $\mathbf{X}$ to be the set of all $X$ such that $\PR( \mathbf{X} = X )$ is non-zero, thus $\mathbf{X}$ almost surely takes values in its essential range.  We also employ the following conditional expectation notation.  If $E$ is an event of non-zero probability, we write
$$ \PR( F | E ) := \frac{\PR( F \wedge E )}{\PR(E)}$$
for any event $F$, and
$$ \E( \mathbf{X} | E ) := \frac{\E({\mathbf X} 1_E)}{\PR(E)}$$
for any (absolutely integrable) real-valued random variable ${\mathbf X}$.  If $\mathbf{Y}$ is another random variable taking at most countably many values, we define the conditional probability $\PR(F|\mathbf{Y})$ to be the random variable that equals $\PR(F|\mathbf{Y}=Y)$ on the event $\mathbf{Y}=Y$ for each $Y$ in the essential range of $\mathbf{Y}$, and similarly define the conditional expectation $\E( \mathbf{X} | \mathbf{Y} )$ to be the random variable that equals $\E( \mathbf{X} | \mathbf{Y} = Y)$ on the event $\mathbf{Y}=Y$.  We observe the idempotency property
\begin{equation}\label{idem}
\E ( \E({\mathbf X}|{\mathbf Y}) ) = \E \mathbf{X}
\end{equation}
whenever ${\mathbf X}$ is absolutely integrable and $\mathbf{Y}$ takes at most countably many values.

We will rely frequently on the following simple concentration of measure result.

\begin{lem}[Chebyshev inequality]\label{cheb}  Let ${\mathbf X}, {\mathbf Y}$ be independent random variables taking at most countably many values.  Let ${\mathbf Y}'$ be a conditionally independent copy of ${\mathbf Y}$ over ${\mathbf X}$; in other words, for every $X$ in the essential range of $\mathbf{X}$, the random variables ${\mathbf Y}, {\mathbf Y}'$ are independent and identically distributed after conditioning to the event $\mathbf{X}=X$.  Let $F( {\mathbf X}, {\mathbf Y})$ be a (absolutely integrable) random variable depending on ${\mathbf X}$ and ${\mathbf Y}$.  Suppose that one has the bounds
\begin{equation}\label{1-moment}
\E F( {\mathbf X}, {\mathbf Y} ) = \alpha + O(\eps \alpha)
\end{equation}
and
\begin{equation}\label{2-moment}
\E F( {\mathbf X}, {\mathbf Y} ) F( {\mathbf X}, {\mathbf Y}' )  = \alpha^2 + O(\eps \alpha^2)
\end{equation}
for some $\alpha, \eps > 0$ with $\eps = O(1)$.  Then for any $\theta > 0$, one has
\begin{equation}\label{conclusion}
\E( F( {\mathbf X}, {\mathbf Y} ) | {\mathbf X} ) = \alpha + O_{\leq}(\theta)
\end{equation}
with probability $1 - O( \frac{\eps\alpha^2}{\theta^2})$.
\end{lem}

In practice, we will often establish \eqref{1-moment} and \eqref{2-moment} by first computing the conditional expectations
$$ \E(F( {\mathbf X}, {\mathbf Y} )|{\mathbf Y})$$
and
$$ \E(F( {\mathbf X}, {\mathbf Y} ) F( {\mathbf X}, {\mathbf Y}' )|{\mathbf Y}, {\mathbf Y}')$$
and then using \eqref{idem}.  Thus we see that we can control the ${\mathbf X}$-conditional expectation of $F( {\mathbf X}, {\mathbf Y} )$ via the ${\mathbf Y}$-conditional expectation, provided that we can similarly control the ${\mathbf Y}, {\mathbf Y}'$-conditional expectation of $F( {\mathbf X}, {\mathbf Y} ) F( {\mathbf X}, {\mathbf Y}' )$.

\begin{proof}  Let ${\mathbf Z}$ denote the random variable
$$ {\mathbf Z} := \E( F( {\mathbf X}, {\mathbf Y} ) | {\mathbf X} )$$
then by the conditional independence and identical distribution of ${\mathbf Y}, {\mathbf Y}'$ over ${\mathbf X}$ we have
$$ {\mathbf Z}^2 = \E( F( {\mathbf X}, {\mathbf Y} ) F( {\mathbf X}, {\mathbf Y}' ) | {\mathbf X} ).$$

From \eqref{1-moment} and \eqref{idem} we have
$$ \E {\mathbf Z} = \alpha + O(\eps \alpha)$$
while from \eqref{2-moment}, \eqref{idem} we have
$$ \E {\mathbf Z}^2 = \alpha^2 + O(\eps \alpha^2)$$
and thus
$$ \E |{\mathbf Z}-\alpha|^2 \ll \eps \alpha^2.$$
The claim now follows from Markov's inequality (or the Chebyshev inequality).
\end{proof}

\section{Sieving a set of primes}\label{sec:initial}

 We begin by using a variant of the
Westzynthius-Erd\H{o}s-Rankin method to reduce this problem to a problem of sieving a set $\QQ$ of \emph{primes} in $[y] \backslash [x]$,
rather than integers in $[y] \backslash [x]$. 

Given a large real number $x$, define
\be\label{ydef}
y := c x \frac{\log x \log_3 x}{\log_2 x},
\ee
where $c$ is a certain (small) fixed positive constant.  Also let
\begin{equation}\label{zdef}
 z :=x^{\log_3 x/(4\log_2 x)},
\end{equation}
and introduce the three disjoint sets of primes
\begin{align}
\cS &:= \{ s\; \mbox{prime} : \log^{20} x < s \le z \},\label{s-def}\\
\PP &:= \{ p \; \mbox{prime}: x/2 < p \leq x\},\label{p-def}\\
\QQ &:= \{ q \;\mbox{prime}: x < q \leq y\}\label{q-def}.
\end{align}
For residue classes $\vec a = (a_s \mod s)_{s\in \cS}$ and
$\vec b = (b_p \mod p)_{p\in \PP}$, define the sifted sets
$$ S(\vec a) := \{ n \in \Z: n \not\equiv a_s \pmod s \hbox{ for all } s \in \cS \}$$
and likewise
$$ S(\vec b) := \{ n \in \Z: n \not\equiv b_p \pmod p \hbox{ for all } p \in \PP \}.$$

We then have

\begin{thm}[Sieving primes]\label{sieve-primes}  Let $x$ be sufficiently
  large and suppose that $y$ obeys  \eqref{ydef}.
Then there are vectors $\vec{a}=(a_s \mod s)_{s \in \cS}$
and $\vec{b}=(b_p \mod p)_{p \in \PP}$, such that
\begin{equation}\label{up-short-random}
 \#(\QQ \cap S( \vec{ a} ) \cap S( \vec{ b} ) ) 
\ll \frac{x}{\log x}.
\end{equation}
\end{thm}

We prove Theorem \ref{sieve-primes} in subsequent sections. 
Here, we show how this theorem implies \eqref{Y-lower}, and hence Theorem \ref{mainthm}.

Let $\vec a$ and $\vec b$ be as in Theorem \ref{sieve-primes}.
We extend the tuple $\vec a$ to a tuple $(a_p)_{p \leq x}$ of congruence classes $a_p \mod p$ for all primes $p \leq x$ by setting $a_p := b_p$ for $p \in \PP$ and $a_p := 0$ for $p \not \in \cS \cup \PP$, and consider the sifted set
$$ {\mathcal T} := \{ n \in [y]\backslash [x] : n \not\equiv a_p \pmod p \hbox{ for all } p \leq x \}.$$
The elements of ${\mathcal T}$, by construction, are not divisible by any prime in $(0,\log^{20} x]$ or in $(z,x/2]$.
 Thus, each element must either be a $z$-smooth number (i.e., a number with all prime factors at most $z$), or must consist of a prime greater than $x/2$, possibly multiplied by some additional primes that are all at least $\log^{20} x$.  However, from \eqref{ydef} we know that $y=o(x\log x)$.
Thus, we see that an element of ${\mathcal T}$ is either a $z$-smooth number or a prime in $\QQ$.  In the second case, the element lies in $\QQ  \cap S(\vec a) \cap S(\vec b)$.  Conversely, every element of $\QQ \cap S(\vec a) \cap S(\vec b)$ lies in ${\mathcal T}$.  Thus, ${\mathcal T}$ only differs from $\QQ  \cap S(\vec a) \cap S(\vec b)$ by a set $\cR$ consisting of $z$-smooth numbers in $[y]$.

To estimate $\# \cR$, let
$$ u := \frac{\log y}{\log z},$$
so from \eqref{ydef}, \eqref{zdef} one has $u \asym 4 \frac{\log_2 x}{\log_3 x}$.  
By standard counts for smooth numbers (e.g. de Bruijn's theorem \cite{deB}) and \eqref{ydef}, we thus have
\[
\# \cR  \ll  y e^{-u\log u + O( u \log\log(u+2) ) } \\
= \frac{y}{\log^{4+o(1)} x} 
= o\( \frac{x}{\log x}\).
\]
Thus, we find that $\# \mathcal T \ll x/\log x$.  

Next, let $C$ be a sufficiently large constant such that
$\# \mathcal T$ is less than the number of primes in $(x,Cx]$.  By matching
each of these surviving elements to a distinct prime in $(x,Cx]$ and choosing congruence classes appropriately, we thus find congruence classes $a_p \mod p$ 
for $p\le Cx$ which cover \emph{all} of the integers in $(x,y]$.  In the language of Definition  \ref{y-def}, we thus have
$$ Y(Cx) \geq y-x+1,$$
and \eqref{Y-lower} follows from \eqref{ydef}.

\begin{remark} One can replace the appeal to de Bruijn's theorem here by the simpler bounds of Rankin \cite[Lemma II]{R1}, if one makes the very minor change of increasing the $4$ in the denominator of \eqref{zdef} to $5$, and also makes similar numerical changes to later parts of the argument.
\end{remark}

It remains to establish Theorem \ref{sieve-primes}.  This is the objective of the remaining sections of the paper.

\section{Using a hypergraph covering theorem}\label{sec:pip}

In the previous section we reduced matters to obtaining residue classes $\vec{a}$, $\vec{b}$ such that the sifted set $\QQ \cap S(\vec{a}) \cap S(\vec{ b})$ is small.  In this section we use a hypergraph covering theorem, generalizing a result of Pippenger and Spencer \cite{pippenger}, to reduce the task to that of finding residue classes $\vec{b}$ that have large intersection with $\QQ \cap S(\vec{a})$.

\subsection{Heuristic discussion}

Consider the following general combinatorial problem.  Let $(V,E_i)_{i \in I}$ be a collection of (non-empty) hypergraphs on a fixed finite vertex set $V$ indexed by some finite index set $I$.  In other words, $V$ and $I$ are finite sets, and for each $i \in I$, $E_i$ is a (non-empty) collection of subsets of $V$.  The problem is then to select a single edge $e_i$ from each set $E_i$ in such a way that the union $\bigcup_{i \in I} e_i$ covers as much of the vertex set $V$ as possible.  (In the context considered in \cite{pippenger}, one considers choosing many edges from a single hypergraph $(V,E)$, which in our context would correspond to the special case when $(V,E_i)$ was independent of $i$.)  One should think of the set $V \backslash \bigcup_{i \in I} e_i$ as a sifted version of $V$, with each $e_i$ representing one step of the sieve.

One simple way to make this selection is a random one: one chooses a random edge $\mathbf{e}_i$ uniformly at random from $E_i$, independently in $i$.  In that case, the probability that a given vertex $v \in V$ survives the sifting (that is, it avoids the random union $\bigcup_{i \in I} \mathbf{e}_i$) is equal to
$$ \prod_{i \in I} (1 - \PR( v \in \mathbf{e}_i )).$$
In applications, the index set $I$ is large and the probabilities $\PR( v \in \mathbf{e}_i )$ are small, in which case the above expression may be well approximated by
$$ \exp( - d_I(v) )$$
where we define the \emph{normalized degree} $d_I(v)$ of $v$ to be the quantity
$$ d_I(v) := \sum_{i \in I} \PR( v \in \mathbf{e}_i ).$$
If we make the informal uniformity assumption 
\begin{itemize}
\item[(i)] One has $d_I(v) \approx d$ for all (or almost all) vertices $v$,
\end{itemize}
we thus expect the sifted set $V \backslash  \bigcup_{i \in I} \mathbf{e}_i$ to have density approximately $\exp(-d)$.

Can one do better than this?  Choosing the $\mathbf{e}_i$ independently is somewhat inefficient because it allows different random edges $\mathbf{e}_i, \mathbf{e}_j$ to collide with each other.  If we could somehow modify the coupling between the $\mathbf{e}_i$ so that they were always disjoint, then the probability that a given vertex $v \in V$ survives the sieve would now become
$$ 1 - \sum_{i \in I} \PR( v \in \mathbf{e}_i ) = 1 - d_I(v).$$
This suggests that one could in principle lower the density of the sifted set from $\exp(-d)$ to $1-d$ (or $\max(1-d,0)$, since the density clearly cannot be negative), and in particular to sift out $V$ almost completely as soon as $d$ exceeds $1$.   

Suppose for the moment that such an optimal level of sieve efficiency is possible, and return briefly to consideration of Theorem \ref{sieve-primes}.  We set the vertex set $V$ equal to $\QQ \cap S(\vec a)$ for some suitable choice of $\vec a$. If we set
$$ y := c x \frac{\log x \log_3 x}{\log_2 x}$$
for some small $c>0$ (in accordance with \eqref{ydef}), then standard probabilistic heuristics (together with Mertens' theorem and \eqref{ydef}, \eqref{s-def}) suggest that $V$ should have cardinality about
$$ \frac{y}{\log x} \times \prod_{s \in \cS} \(1-\frac{1}{s}\) \approx c \frac{x}{\log x} \log_2 x,$$
so in particular this set is roughly $c \log_2 x$ times larger than $\PP$.  In later sections, we will use the multidimensional 
sieve from \cite{maynard-large}, \cite{maynard-dense} to locate for most primes $p$ in $\PP$, a large number of residue classes $b_p \mod p$ that each intersect $\QQ \cap S(\vec a)$ in roughly $\order \log_2 x$ elements on the average.  If we let $E_p$ be the set of all such intersections $(b_p \mod p) \cap V$, then the task of making $\QQ \cap S(\vec a) \cap S(\vec b)$ small is essentially the same as making the sifted set $V \backslash \bigcup_{p \in \PP} e_p$ small, for some suitable edges $e_p$ drawn from $E_p$.  By double counting, the expected density $d$ here should be roughly
$$ d \order \frac{\# \PP \times \log_2 x}{\# V} \order \frac{1}{c},$$
and so one should be able to sieve out $\QQ \cap S(\vec a)$ more or less completely once $c$ is small enough if one had optimal sieving.  In contrast, if one used independent sieving, one would expect the cardinality of $\QQ \cap S(\vec a) \cap S(\vec b)$ to be something like $\exp( - 1/c ) \times c \frac{x}{\log x} \log_2 x$, which would only be acceptable if $c$ was slightly smaller than $\frac{1}{\log_3 x}$.  This loss of $\log_3 x$ ultimately leads to the loss of $\log_4 X$ in \eqref{G1-bound} as compared against Theorem \ref{mainthm}.

It is thus desirable to obtain a general combinatorial tool for achieving near-optimal sieve efficiency for various collections $(V,E_i)_{i \in I}$ of hypergraphs.
The result of Pippenger and Spencer \cite{pippenger} (extending previous results of R\"odl \cite{nibble} and Frankl and R\"odl \cite{rodl}, as well as unpublished work of Pippenger) asserts, very roughly speaking, that one can almost attain this optimal efficiency under some further assumptions beyond (i), which we state informally as follows:
\begin{itemize}
\item[(ii)] The hypergraphs $(V,E_i)$ do not depend on $i$.
\item[(iii)]  The \emph{normalized codegrees} $\sum_{i \in I} \PR( v, w \in \mathbf{e}_i )$ for $v \neq w$ are small.
\item[(iv)] The edges $e_i$ of $E_i$ are of \emph{constant} size, thus there is a $k$ such that $\# e_i = k$ for all $i$ and all $e_i \in E_i$.
\end{itemize}
The argument is based on the \emph{R\"odl nibble} from \cite{nibble}, which is a variant of the \emph{semi-random method} from \cite{aks}.  Roughly speaking, the idea is to break up the index set $I$ into smaller pieces $I_1,\dots,I_m$.  For the first $I_1$, we perform a ``nibble'' by selecting the $\mathbf{e}_i$ for $i \in I_1$ uniformly and independently.  For the next nibble at $I_2$, we restrict (or condition) the $\mathbf{e}_i$ for $i \in I_2$ to avoid the edges arising in the first nibble, and \emph{then} select $\mathbf{e}_i$ for $i \in I_2$ independently at random using this conditioned distribution.  We continue performing nibbles at $I_3,\dots,I_m$ (restricting the edges at each nibble to be disjoint from the edges of previous nibbles) until the index set $I$ is exhausted.  Intuitively, this procedure enjoys better disjointness properties than the completely independent selection scheme, but it is harder to analyze the probability of success.  To achieve the latter task, Pippenger and Spencer rely heavily on the four hypotheses (i)-(iv).

In our context, hypothesis (iii) is easily satisfied, and (i) can also be established.  Hypothesis (ii) is not satisfied (the $E_p$ vary in $p$), but it turns out that the argument of Pippenger and Spencer can easily be written in such a way that this hypothesis may be discarded.  
But it is the failure of hypothesis (iv) which is the most severe difficulty: the size of the sets $e_p = (b_p \mod p) \cap V$ can fluctuate quite widely for different choices of $p$ or $b_p$.  This creates an undesirable bias in the iterative nibbling process: with each nibble, larger edges $\mathbf{e}_i$ have a reduced chance of survival compared with smaller edges, simply because they have more elements that could potentially intersect previous nibbles.  Given that one expects the larger edges to be the most useful for the purposes of efficient sieving, this bias is a significant problem.  One could try to rectify the issue by partitioning the edge sets $E_i$ depending on the cardinality of the edges, and working on one partition at a time, but this seriously impacts hypothesis (i) in a manner that we were not able to handle.

Our resolution to this problem is to modify the iterative step of the nibbling process by \emph{reweighting} the probability distribution of the $\mathbf{e}_i$ at each step to cancel out the bias incurred by conditioning an edge $\mathbf{e}_i$ to be disjoint from previous nibbles.  It turns out that there is a natural choice of reweighting for this task even when the normalized degrees $d_I(v)$ vary in $v$.  As a consequence, we can obtain a version of the Pippenger-Spencer theorem in which hypothesis (ii) is essentially eliminated and (i), (iv) significantly weakened, leaving only (iii) as the main hypothesis.  We remark that a somewhat similar relaxation of hypotheses (i)-(iv) was obtained by Kahn in \cite{kahn}, although the statement in \cite{kahn} is not exactly in a form convenient for our applications here.

\subsection{Statement of covering theorem}

We now rigorously state the hypergraph covering theorem that we will use.  In order to apply this theorem for our application, we will need a probabilistic formulation of this theorem which does not, at first glance, bear much resemblance to the combinatorial formulation appearing in \cite{pippenger}; we will discuss the connections between these formulations shortly.  We will also phrase the theorem in a completely quantitative fashion, avoiding the use of asymptotic notation; this will be convenient for the purposes of proving the theorem via induction (on the number $m$ of ``nibbles'').

\begin{thm}[Probabilistic covering]\label{packing-quant}  There exists
  a constant $C_0 \geq 1$ such that the following holds.  Let $D, r, A
  \geq 1$, $0 < \kappa \leq 1/2$, and let $m \geq 0$ be an integer.
  Let $\delta > 0$ satisfy the smallness bound
\begin{equation}\label{delta-small}
\delta \leq \left( \frac{\kappa^A}{C_0 \exp(AD)} \right)^{10^{m+2}}.
\end{equation}
Let $I_1,\dots,I_m$ be disjoint finite non-empty sets, and let $V$ be a finite set.  For each $1 \leq j \leq m$ and $i \in I_j$, let $\mathbf{e}_i$ be a random finite subset of $V$.  Assume the following:
\begin{itemize}
\item (Edges not too large) Almost surely for all $j=1,\dots,m$ and $i \in I_j$, we have
\begin{equation}\label{p-bound-quant}
 \# \mathbf{e}_i \leq r;
\end{equation}
\item (Each sieve step is sparse) For all $j=1,\dots,m$, $i \in I_j$ and $v \in V$,
\begin{equation}\label{q-form-quant}
\PR( v \in \mathbf{e}_i ) \leq \frac{\delta}{(\# I_j)^{1/2}};
\end{equation}
\item (Very small codegrees) For every $j=1,\dots,m$, and distinct $v_1,v_2 \in V$,
\begin{equation}\label{q-form-2-quant}
\sum_{i \in I_j} \PR( v_1,v_2 \in \mathbf{e}_i ) \leq \delta
\end{equation}
\item (Degree bound) If for every $v \in V$ and $j =1,\dots,m$ we introduce the normalized degrees
\begin{equation}\label{epsqj-quant}
 d_{I_j}(v) := \sum_{i \in I_j} \PR( v \in \mathbf{e}_i )
\end{equation}
and then recursively define the quantities $P_j(v)$ for $j=0,\dots,m$ and $v \in V$ by setting
\begin{equation}\label{p0-def}
P_0(v) := 1 
\end{equation}
and
\begin{equation}\label{pj-def}
P_{j+1}(v) := P_j(v) \exp( - d_{I_{j+1}}(v) / P_j(v) )
\end{equation}
for $j=0,\dots,m-1$ and $v \in V$, then we have
\begin{equation}\label{epsbound-quant}
d_{I_j}(v) \leq D P_{j-1}(v) \qquad (1 \le j\le m, v\in V)
\end{equation}
 and
\begin{equation}\label{pje-size-bite}
P_j(v) \geq \kappa \qquad (0 \le j\le m, v\in V).
\end{equation}
\end{itemize}
Then we can find random variables $\mathbf{e}'_i$ for each $i \in \bigcup_{j=1}^m I_j$ with the following properties:
\begin{itemize}
\item[(a)]  For each $i \in \bigcup_{j=1}^m I_j$, the essential support of $\mathbf{e}'_i$ is contained in the essential support of $\mathbf{e}_i$, union the empty set singleton $\{\emptyset\}$.  In other words, almost surely $\mathbf{e}'_i$ is either empty, or is a set that $\mathbf{e}_i$ also attains with positive probability.
\item[(b)]  For any $0 \leq J \leq m$ and any finite subset $e$ of $V$ with $\# e \leq A - 2rJ$, one has
\begin{equation}\label{penj-quant}
 \PR\left( e \subset V \backslash \bigcup_{j=1}^J \bigcup_{i \in I_j} \mathbf{e}'_i \right)
 = \(1 + O_{\leq}( \delta^{1/10^{J+1}} )\) P_J(e) 
\end{equation}
where
\begin{equation}\label{pje-def}
P_j(e) := \prod_{v \in e} P_j(v).
\end{equation}

\end{itemize}
\end{thm}

We prove this theorem in Section \ref{pack-proof}.  It is likely that
the smallness condition \eqref{delta-small} can be relaxed, for
instance by modifying the techniques from \cite{vanvu}.  However, this
would not lead to any significant improvement in the final bound on
$G(X)$ in Theorem \ref{mainthm}, as in our application the condition \eqref{delta-small} is already satisfied with some room to spare.  The parameter $r$ does not appear explicitly in the smallness requirement \eqref{delta-small}, but is implicit in that requirement since the conclusion is trivially true unless $2r < A$.

We now discuss some special cases of this theorem which are closer to the original hypergraph covering lemma of Pippenger and Spencer. (Readers who are interested only in large gaps between primes can skip ahead to Section \ref{apply}.)  If $(V,E)$ is a hypergraph, we can take each $\mathbf{e}_i$ to be an edge of $E$ drawn uniformly at random.  If $I_j$ has cardinality $n_j$, we obtain the following corollary:

\begin{cor}[Combinatorial covering]\label{comcover}
There exists a constant $C_0 \geq 1$ such that the following holds.  Let $D, r \geq 1$, $0 < \kappa \leq 1/2$, and let $m \geq 0$ and $n_1,\ldots,n_m \geq 1$ be integers.  Set $A := 2rm+1$, and let $\delta > 0$ be a quantity obeying the smallness condition \eqref{delta-small}.  Let $(V,E)$ be a hypergraph, and assume the following axioms:
\begin{itemize}
\item[(i)] All edges $e$ in $E$ have cardinality at most $r$.
\item[(ii)] For every $v \in V$, the degree $\operatorname{deg}(v) := \# \{ e \in E: v \in e \}$ is at most $\frac{\delta}{n^{1/2}} \# E$, $n=\min(n_1,\ldots,n_m)$.
\item[(iii)] For every distinct $v,w \in V$, the codegree $\operatorname{codeg}(v,w) := \{ e \in E: v,w \in e \}$ is at most $\frac{\delta}{n} \# E$.
\item[(iv)] If for every $v \in V$ we introduce the $P_j(v)$ for $j=0,\dots,m$ by
$$ P_0(v) := 1$$
and 
$$ P_{j+1}(v) = P_j(v) \exp\( - \frac{n_j \operatorname{deg}(v)}{(\# E) P_j(v)} \)
\qquad (1\le j\le m),
$$
then we have
\begin{equation}\label{v1}
 \frac{n_{j+1} \operatorname{deg}(v)}{(\# E) P_j(v)} \leq D \qquad (0\le j\le m-1)
\end{equation}
and
\begin{equation}\label{v2}
 P_m(v) \geq \kappa.
\end{equation}
\end{itemize}
Then we can find edges $e_1,\dots,e_l \in E$ with $l \leq n_1+\cdots+n_m$ 
such that
$$
\# ( V \backslash (e_1 \cup \dots \cup e_l) ) \ll \sum_{v \in V} P_m(v).$$
\end{cor}

\begin{proof}  Let $N=n_1+\cdots+n_m$.
 By Theorem \ref{packing-quant} (with $\mathbf{e}_i$ and $I_j$ as indicated above), we may find random variables $\mathbf{e}'_i$ for $i=1,\dots,N$ taking values in $E \cup \{\emptyset\}$ such that
$$ \PR\bigg( v \subset V \backslash \bigcup_{i=1}^{N} \mathbf{e}'_i \bigg) = \(1 + O_{\leq}\( \delta^{1/10^{m+1}} \)\) P_m(v) 
$$
for each $v \in V$, and in particular by linearity of expectation
$$ \E \#\(V \backslash \bigcup_{i=1}^{N} \mathbf{e}'_i\) \ll \sum_{v \in V} P_m(v).$$
Thus we can find instances $e'_i$ of $\mathbf{e}'_i$ such that
$$ \#\(V \backslash \bigcup_{i=1}^{N} e'_i\) \ll \sum_{v \in V} P_m(v).$$
Discarding the empty edges $e'_i$, we obtain the claim.
\end{proof}

We now give a qualitative version of the above corollary, in which all objects involved can depend on asymptotic parameter $x$ going to infinity:

\begin{cor}[Generalized Pippenger-Spencer]
Let $(V,E)$ be a hypergraph, and let $d \geq 1$ be a quantity obeying the following:
\begin{itemize}
\item One has $d = o(\# E)$.
\item All edges $e$ in $E$ have cardinality $O(1)$.
\item For every $v \in V$, one has $d \leq \operatorname{deg}(v) \ll d$.
\item For every distinct $v,w \in V$, one has $\operatorname{codeg}(v,w) = o(d)$.
\end{itemize}
Then we can find edges $e_1,\dots,e_l \in E$ with $l \lesssim \frac{\# E}{d}$ such that
\begin{equation}\label{vel}
\# ( V \backslash (e_1 \cup \dots \cup e_l) ) = o( \# V ).
\end{equation}
\end{cor}

Note that for any given vertex $v$, the probability that a randomly selected edge $\mathbf{e}$ from $E$ will cover $v$ is $\frac{\operatorname{deg}(v)}{\# E}$, which is roughly $\frac{d}{E}$.  Thus the conclusion of the above corollary uses an essentially optimal number of edges.

\begin{proof}  By a diagonalization argument, it suffices for any fixed $\eps>0$ (independent of $x$) to show that one can find edges $e_1,\dots,e_l \in E$ with $l \leq (1+\eps) \frac{\# E}{d}$ such that 
$$ \# ( V \backslash (e_1 \cup \dots \cup e_l) ) \ll \eps \# V$$
for sufficiently large $x$.

Let $k \geq 1$ be a fixed integer (depending on $\eps$) to be chosen later, and
let $m=k^2$.  For $x$ large enough, we can find a natural number $n_1$ such that
\begin{equation}\label{1e}
\frac{\# E}{d} \leq n_1 k \leq \(1+\frac{\eps}2\) \frac{\# E}{d},
\end{equation}
and we define
\[
n_j = \lceil n_1 e^{(1-j)/k} \rceil \qquad (1\le j\le m).
\]
We now verify the conditions of Corollary \ref{comcover} with suitable choices of parameters $D,r,\kappa,\delta$.  Clearly (i) is obeyed with $r=O(1)$, and a short computation reveals that (ii), (iii) are obeyed for some $\delta = o(1)$, and (iv) is obtained for some $D = O(1)$ and $\kappa \gg 1$.  Applying Corollary \ref{comcover}, we may thus find (for $x$ sufficiently large) $e_1,\dots,e_l \in E$ with $l\leq n_1+\cdots n_m$ such that
$$ \# ( V \backslash (e_1 \cup \dots \cup e_l) ) \ll \sum_{v \in V} P_m(v).$$
We have
\[
n_1+\cdots+n_m \le m + n_1 \sum_{j=1}^\infty e^{(1-j)/k} = m+ \frac{n_1}{1-e^{-1/k}} = k^2+n_1(k+O(1)) \le (1+\eps) \frac{\# E}{d}
\]
by first taking $k$ large enough, then $x$ large enough.
Next, an easy induction shows that $P_j(v) \leq p_j$ for all $j=0,\dots,m$, where $p_0 := 1$ and
$$ p_{j+1} := p_j \exp\( - \frac{n_{j+1} d}{(\# E) p_j} \) \qquad (0\le j\le m-1).
$$
Another easy induction using \eqref{1e} shows that $p_j\le e^{-j/k}$ for all
$0\le j\le m$.  In particular, $p_m\le e^{-k} \le \eps $ if $k\ge \log(1/\eps)$.
\end{proof}

\subsection{Applying the covering theorem}\label{apply}
%
%

We now specialize Theorem \ref{packing-quant} to a situation relevant for the application to large prime gaps. 

\begin{cor}\label{packing-quant-cor}  Let $x\to\infty$.
Let $\PP'$, $\QQ'$ be sets 
with $\# \PP' \le x$ and $\#\QQ' > (\log_2 x)^3$.
For each $p \in \PP'$, let $\mathbf{e}_p$ be a random subset of $\QQ'$
satisfying the size bound
\be\label{rbound}
\# \mathbf{e}_p \le r = O\( \frac{\log x \log_3 x}{\log_2^2 x} \) \qquad
(p\in \PP').
\ee
Assume the following:
\begin{itemize}
\item (Sparsity) For all $p \in \PP'$ and $q \in \QQ'$,
\begin{equation}\label{q-form-quant-cor}
\PR( q \in \mathbf{e}_p ) \le x^{-1/2 - 1/10}.
\end{equation}
\item (Uniform covering) For all but at most $\frac{1}{(\log_2 x)^2} \# \QQ'$ elements $q \in \QQ'$, we have
\begin{equation}\label{pje-size-bite-cor}
\sum_{p \in \PP'} \PR( q \in \mathbf{e}_p) = C + O_{\le}\pfrac{1}{(\log_2 x)^2}
 \end{equation}
for some quantity $C$, independent of $q$, satsifying
\begin{equation}\label{sigma}
\frac{5}{4} \log 5 \le C \ll 1.
\end{equation}
\item (Small codegrees) For any distinct $q_1,q_2 \in \QQ'$,
\be\label{small-codegree-cor}
\sum_{p\in\PP'} \PR(q_1,q_2\in \mathbf{e}_p) \le x^{-1/20}.
\ee
\end{itemize}
Then for any positive integer $m$ with
\begin{equation}\label{moo}
 m \le \frac{\log_3 x}{\log 5},
\end{equation}
we can find random sets $\mathbf{e}'_p\subseteq \QQ'$ for each $p \in \PP'$ such that
\[
\# \{ q \in \QQ':  q \not\in \mathbf{e}'_p \hbox{ for all } p \in \PP' \} \sim 5^{-m} \# \QQ'
\]
with probability $1-o(1)$.  More generally, for any $\QQ'' \subset \QQ'$ with cardinality at least $(\# \QQ')/\sqrt{\log_2 x}$, one has
\[
\# \{ q \in \QQ'':  q \not\in \mathbf{e}'_p \hbox{ for all } p \in \PP' \} \sim 5^{-m} \# \QQ''
\]
with probability $1-o(1)$.  The decay rates in the $o(1)$ and $\sim$ notation
are uniform in $\PP'$, $\QQ'$, $\QQ''$.
\end{cor}

For the arguments in this paper, we only need the case $\QQ'' = \QQ'$, but the more general situation $\QQ'' \subset \QQ'$ will be of use in the sequel \cite{FMT} of this paper when we consider chains of large gaps.

\begin{proof}  It suffices to establish the claim for $x$ sufficiently large, as the claim is trivial for bounded $x$.
The number of exceptional elements $q$ of $\QQ'$ that fail \eqref{pje-size-bite-cor} is $o(5^{-m} \# \QQ'')$, thanks to \eqref{moo}.  Thus we may discard these elements from $\QQ'$ and assume that \eqref{pje-size-bite-cor} holds for \emph{all} $q \in \QQ'$, since this does not significantly affect the conclusions of the corollary.

By \eqref{sigma}, we may find disjoint intervals 
$\mathscr{I}_1,\dots,\mathscr{I}_m$ in $[0,1]$ with length 
\begin{equation}\label{ij}
|\mathscr{I}_j| = \frac{5^{1-j}\log 5}{C}
\end{equation}
for $j=1,\dots,m$.
Let $\vec{\mathbf{t}} = (\mathbf{t}_p)_{p \in \PP'}$ be a tuple of elements $\mathbf{t}_p$ of $[0,1]$ drawn uniformly and independently at random for each $p \in \PP'$ (independently of the $\mathbf{e}_p$), and define the random sets
$$ I_j = I_j( \vec{\mathbf{t}} ) := \{ p \in \PP': \mathbf{t}_p \in \mathscr{I}_j \}$$
for $j=1,\dots,m$.  These sets are clearly disjoint.

We will verify (for a suitable choice of $\vec{\mathbf{t}}$) the hypotheses of Theorem \ref{packing-quant} with the indicated sets $I_j$ and random variables $\mathbf{e}_p$, and with suitable choices of parameters $D, r, A \geq 1$ and $0 < \kappa \leq 1/2$, and $V=\QQ'$.

Set
\begin{equation}\label{deltaa}
 \delta := x^{-1/20}
\end{equation}
and observe from \eqref{q-form-quant-cor} that (if $x$ is sufficiently large) one has
\be\label{q=np}
 \PR(  q \in \mathbf{e}_p  ) \leq \frac{\delta}{(\# I_j)^{1/2}}
\ee
for all $j=1,\dots,m$, $p \in I_j$, and $q \in \QQ'$.  
Clearly the small
codegree condition \eqref{small-codegree-cor} implies that
\be\label{q1q2=np}
 \sum_{p \in I_j} \PR( q_1, q_2  \in \mathbf{e}_p  ) \leq \delta
\qquad (1\le j\le m).\ee

Let $q \in \QQ'$, $1\le j\le m$ and consider the independent random
variables $(\mathbf{X}_p^{(q,j)}(\vec{\mathbf{t}}))_{p\in \PP'}$, where
$$
\mathbf{X}_p^{(q,j)}(\vec{\mathbf{t}})=\begin{cases}
\PR( q \in \mathbf{e}_p) & \text{ if } p\in I_j \\
0 & \text{ otherwise.}
\end{cases}
$$
By \eqref{pje-size-bite-cor}, \eqref{sigma} and \eqref{ij}, for every
$j$ and every $q\in \QQ'$,
\[
\sum_{p\in \PP'}\E \mathbf{X}_p^{(q,j)}(\vec{\mathbf{t}}) = 
\sum_{p\in \PP'} \PR( q \in \mathbf{e}_p) \PR(p\in I_j(\vec{\mathbf{t}}))
= |\mathscr{I}_j| \sum_{p\in \PP'}
\PR( q \in \mathbf{e}_p) = 5^{1-j}\log 5 + O_{\le}\pfrac{4/5}{(\log_2 x)^2}.
\]
By \eqref{q-form-quant-cor}, we have
$|\mathbf{X}_p^{(q,j)}(\vec{\mathbf{t}})| \le x^{-1/2-1/10}$ for all $p$, and hence by
Hoeffding's inequality,
\begin{align*}
\PR\( \Big| \sum_{p\in\PP'} (\mathbf{X}_p^{(q,j)}(\vec{\mathbf{t}})-\E \mathbf{X}_p^{(q,j)}(\vec{\mathbf{t}})) \Big| \ge
\frac{1}{(\log_2 x)^2} \) &\le 2 \exp \left\{ - \frac{(\log_2 x)^{-4}}{2x^{-1-1/5} \# I_j} \right\} \\ &\le 2  \exp \left\{ - \frac{x^{1/5}}{(\log_2 x)^4} \right\}
\ll \frac{1}{x^4}.
\end{align*}
By a union bound, there is a deterministic 
choice $\vec{t}$ of $\vec{\mathbf{t}}$ (and hence $I_1,\dots,I_m$)
such that for \emph{every} $q\in \QQ'$ and \emph{every}
$j=1,\ldots,m$, we have
\[
\Big| \sum_{p\in\PP'} (\mathbf{X}_p^{(q,j)}(\vec{t})-\E \mathbf{X}_p^{(q,j)}(\vec{\mathbf{t}})) \Big| < \frac{1}{(\log_2 x)^2}.
\]
We fix this choice $\vec{t}$ (so that the $I_j$ are now deterministic),
and we conclude that
\begin{equation}\label{oo}
\sum_{p\in\PP'} \mathbf{X}_p^{(q,j)}(\vec{t}) = 
\sum_{p \in I_j} \PR( q \in \mathbf{e}_p) =  5^{1-j}\log 5 + O_{\le}\pfrac{2}{(\log_2 x)^2}
\end{equation}
uniformly for all $j=1,\dots,m$, and all $q\in \QQ'$.

From \eqref{epsqj-quant} and \eqref{moo}, we now have
$$ d_{I_j}(q) = (1 + O_\leq(\mu)) 5^{-j+1} \log 5$$
for all $q \in \QQ'$, $1\le j\le m$ and some $|\mu| \le 2/\log_2 x$.  A routine induction using \eqref{p0-def}, \eqref{pj-def} then shows (for $x$ sufficiently large) that
\begin{equation}\label{moldy}
 P_j(q) = (1 + O_{\leq}( 4^j \mu ) ) 5^{-j} = 5^{-j}(1+O_{\le}(2(\log_2 x)^{-\nu})) \qquad (0\le j\le m),
\end{equation}
where $\nu = \log(5/4)/\log(5)$.
In particular we have
$$ d_{I_j}(q) \leq D P_{j-1}(q)\qquad (1\le j\le m)$$
for some $D=O(1)$, and
$$ P_j(q) \geq \kappa \qquad (1\le j\le m),$$
where
$$ \kappa \gg 5^{-m}.$$
We now set
$$ A := 2rm+2.$$
By \eqref{moo} and \eqref{rbound},
$$ A \ll \frac{\log x \log_3^2 x}{\log_2^2 x},$$
and so
$$ \frac{\kappa^A}{C_0 \exp(AD)} \gg \exp\( - O\( \frac{\log x \log_3^3 x}{\log_2^2 x} \) \). $$
By \eqref{moo} and \eqref{deltaa},  we see that
$$ \delta^{1/10^{m+2}} \le \exp\( -\frac{\log{x}}{2000(\log_2{x})^{\log{10}/\log{5}}}\),$$ 
and so  \eqref{delta-small} is satisfied $x$ is large enough (note that $\log 10 / \log 5 < 2$). Thus all the hypotheses of Theorem \ref{packing-quant} have been verified for this choice of parameters (note that $A,\kappa$ and $D$ are independent of $\PP'$, $\QQ'$).

Applying Theorem \ref{packing-quant} (with $V = \QQ'$) and using \eqref{moldy}, one thus obtains random variables $\mathbf{e}'_p$ for $p \in \bigcup_{j=1}^m I_j$ whose essential range is contained in the essential range of $\mathbf{e}_p$ together with $\emptyset$, such that
\be\label{qmiss}
 \PR\left( q \not \in \bigcup_{j=1}^m \bigcup_{p \in I_j} \mathbf{e}'_p \right) = 5^{-m} \(1+O((\log_2 x)^{-\nu})\) 
\ee
for all $q \in \QQ'$, and
\be\label{q1q2miss}
 \PR\left( q_1, q_2 \not \in \bigcup_{j=1}^m \bigcup_{p \in I_j} \mathbf{e}'_p \right) = 5^{-2m} \(1+O((\log_2 x)^{-\nu})\) 
\ee
for all distinct $q_1,q_2 \in \QQ'$.  

Set $\mathbf{e}'_p = \emptyset$ for $p \in \PP' \backslash \bigcup_{j=1}^m I_j$.  
Let $\QQ''$ be as in the corollary, and consider the random variable
$$ \mathbf{Y} := \# \{ q \in \QQ'': q \not\in \mathbf{e}'_p \hbox{ for all } p \in \PP' \}= \sum_{q \in \QQ''} 1_{q \not \in \bigcup_{j=1}^m \bigcup_{p \in I_j} \mathbf{e}'_p}.$$
Using \eqref{qmiss} and \eqref{q1q2miss}, we obtain
$$ \E \mathbf{Y} = 5^{-m} \(1+O((\log_2 x)^{-\nu})\) \# \QQ''$$
and
$$ \E \mathbf{Y}^2 = 5^{-2m} \(1+O((\log_2 x)^{-\nu})\) (\# \QQ'')^2
+O(5^{-m}\# \QQ'')  = 5^{-2m} \(1+O((\log_2 x)^{-\nu})\) (\# \QQ'')^2,$$
(here we use \eqref{moo} and the mild bound $\# \QQ'' > (\log_2 x)^2$), 
and so from Chebyshev's inequality we have
$$ \mathbf{Y} \sim 5^{-m} \# \QQ''$$
with probability $1-o(1)$, as required.
\end{proof}

In view of the above corollary, we may now reduce Theorem \ref{sieve-primes} to the following claim.

\begin{thm}[Random construction]\label{sieve-primes-2}  Let $x$ be a sufficiently large real number and define $y$ by \eqref{ydef}.
 Then there is a quantity $C$ with
\begin{equation}\label{sigma-order}
C \order \frac{1}{c}
\end{equation}
with the implied constants independent of $c$, a tuple of positive integers
$(h_1,\ldots,h_r)$ with $r\le \sqrt{\log x}$,
and some way to choose random vectors $\vec{\mathbf{a}}=(\mathbf{a}_s \mod s)_{s \in \cS}$
and $\vec{\mathbf{n}}=(\mathbf{n}_p)_{p \in \PP}$ of congruence classes $\mathbf{a}_s \mod s$ and integers $\mathbf{n}_p$ respectively, obeying the following:
\begin{itemize}
\item For every $\vec a$ in the essential range of $\vec{\mathbf{a}}$, one has
$$ \PR( q \in \mathbf{e}_p(\vec{a})  | \vec{\mathbf{a}} = \vec a) 
\le x^{-1/2 - 1/10} \quad (p\in \PP),
$$
where $\mathbf{e}_p(\vec{a}):=\{ \mathbf{n}_p+h_i p : 1\le i\le r\} \cap \QQ \cap S(\vec{a})$.
\item With probability $1-o(1)$ we have that
\begin{equation}\label{treat}
 \# (\QQ \cap S(\vec{\mathbf{a}})) \sim 80 c \frac{x}{\log x} \log_2 x.
\end{equation}
\item Call an element $\vec a$ in the essential range of $\vec{\mathbf{a}}$
\emph{good} if, for all but at most 
$\frac{x}{\log x \log_2 x}$ elements $q \in \QQ \cap S(\vec{a})$, one has
\be\label{good}
\sum_{p \in \PP} \PR( q \in \mathbf{e}_p(\vec{a})  | \vec{\mathbf{a}}=\vec{a}) = C + O_{\le} \pfrac{1}{(\log_2 x)^2}.
\ee
Then  $\vec{\mathbf{a}}$ is good with probability $1-o(1)$.
\end{itemize}
\end{thm}

We now show why Theorem \ref{sieve-primes-2} implies Theorem \ref{sieve-primes}.  By \eqref{sigma-order}, we may choose $0 < c < 1/2$ small enough so that \eqref{sigma} holds.  Take 
\[
 m = \left\lfloor \frac{\log_3 x}{\log 5} \right\rfloor.
\]
Now let $\vec{\ba}$ and $\vec{\mathbf{n}}$ be the random vectors
guaranteed by Theorem \ref{sieve-primes-2}.
Suppose that we are in the probability $1-o(1)$ event that $\vec{\mathbf{a}}$ takes a value $\vec a$ which is good and such that \eqref{treat} holds.  
Fix some $\vec a$ within this event.  We may apply Corollary \ref{packing-quant-cor} with $\PP'=\PP$ and $\QQ'=\QQ\cap S(\vec{a})$
for the random variables $\mathbf{n}_p$ conditioned to $\vec{\ba}=\vec{a}$.
A few hypotheses of the Corollary must be verified.
First, \eqref{pje-size-bite-cor} follows from \eqref{good}.
The small codegree condition \eqref{small-codegree-cor} is also quickly
checked.  Indeed, for distinct $q_1,q_2\in \QQ'$, if $q_1,q_2 \in \mathbf{e}_p(\vec{a})$ then $p|q_1-q_2$.  But $q_1-q_2$ is a nonzero integer
of size at most $x\log x$, and is thus divisible by at most one prime $p_0\in \PP'$.  Hence
\[
\sum_{p\in\PP'} \PR(q_1,q_2\in\mathbf{e}_p(\vec{a})) =  \PR(q_1,q_2\in\mathbf{e}_{p_0}(\vec{a})) \le x^{-1/2-1/10},
\]
the sum on the left side being zero if $p_0$ doesn't exist.
By  Corollary \ref{packing-quant-cor}, there exist 
random variables $\mathbf{e}'_p(\vec{a})$, whose essential range is
contained in the essential range of $\mathbf{e}_p(\vec{a})$ together
with $\emptyset$, and satisfying 
$$
 \{ q\in \QQ \cap S(\vec{a}) : q\not\in \mathbf{e}'_p(\vec{a}) \text{ for all }p\in\PP\} 
\sim 5^{-m} \# (\QQ\cap S(\vec{a})) \ll \frac{x}{\log x}
$$
with probability $1-o(1)$, where we have used  \eqref{treat}.
Since $\mathbf{e}'_p(\vec{a})=\{\mathbf{n}'_p+h_i p : 1\le i\le r\} \cap \QQ
\cap S(\vec{a})$ for some random integer $\mathbf{n}'_p$, it follows that
$$
 \{ q\in \QQ \cap S(\vec{a}) : q\not \equiv \mathbf{n}'_p\pmod{p} \text{ for all }p\in\PP\} \ll \frac{x}{\log x}
$$
with probability $1-o(1)$.  Taking a specific $\vec{\mathbf{n}}'=\vec{n}'$ for which this
relation holds and setting
 $b_p=n'_p$ for all $p$ concludes the proof of 
the claim \eqref{up-short-random} and establishes
 Theorem \ref{sieve-primes}.

It remains to establish Theorem \ref{sieve-primes-2}.  This will be achieved in later sections.

\section{Proof of covering theorem}\label{pack-proof}

We now prove Theorem \ref{packing-quant}.  
Let $C_0$ be a sufficiently large absolute constant.

We induct on $m$.  The case $m=0$ is vacuous, so suppose that $m \geq
1$ and that the claim has already been proven for $m-1$.  Let
$D,r,A,\kappa,\delta,I_j,\mathbf{e}_i,V$ be as in the theorem.  By the
induction hypothesis, we can already find random variables
$\mathbf{e}'_i$ for $i \in \bigcup_{j=1}^{m-1} I_j$ obeying the
conclusions (a), (b) of the theorem for $m-1$.  In particular, we may
form the partially sifted set
$$ 
\mathbf{W} := V \backslash \bigcup_{j=1}^{m-1} \bigcup_{i \in I_j} {\mathbf e}'_i,$$
and we have
\begin{equation}\label{hyp}
 \PR( e \subset \mathbf{W} ) = (1 + O_{\leq}( \delta^{1/10^{m}} )) P_{m-1}(e)
\end{equation}
whenever $e\subset V$ has cardinality $\# e \leq A - 2r(m-1)$. 

Our task is then to construct random variables $\mathbf{e}'_i$ for $i
\in I_m$, possibly coupled with existing random variables such as
$\mathbf{W}$, whose essential range is contained in that of
$\mathbf{e}_i$ together with the empty set, and
such that 
\begin{equation}\label{econc-quant}
 \PR\( e \subset \mathbf{W} \backslash \bigcup_{i \in I_m} \mathbf{e}'_i \) = \(1 + O_{\leq}( \delta^{1/10^{m+1}} )\) P_m(e) 
\end{equation}
for all finite subsets $e$ of $V$ with $\# e \leq A - 2rm$.  Note that
we may assume that $A > 2rm$, as the claim \eqref{penj-quant} 
is trivial otherwise.  In particular we have
\begin{equation}\label{e2r}
A - 2r(m-1) > 2r.
\end{equation}

\newcommand{\tei}{\tilde{e}_i}
\newcommand{\hei}{\hat{e}_i}

From \eqref{pje-size-bite}, \eqref{pje-def} we note that
\begin{equation}\label{pje-size-quant}
P_j(\tilde{e}) \geq \kappa^{\# \tilde{e}}
\end{equation}
whenever $j=1,\dots,m$ and all $\tilde{e} \subset V$.
In particular, by \eqref{pje-size-quant} and \eqref{p-bound-quant},
 whenever $\tei$ is in the essential range of $\bfe_i$, we have
\be\label{pjei-size-quant}
P_j(\tei) \geq \kappa^{r}.
\ee
For future reference, we observe that from \eqref{e2r} and
\eqref{delta-small}, we have
\be\label{delta-small-cor}
r \kappa^{-r} \le A \kappa^{-r} \le A^2 \kappa^{-2r} \le A^2 D
\kappa^{-A} \le \delta^{-1/10^{m+2}}.
\ee

For each $i \in I_m$, and every $W$ in the essential range of $\mathbf{W}$, 
define the normalization factor
\begin{equation}\label{xdef}
X_i(W) := 
\E\(\frac{1_{\mathbf{e}_i \subset W}}{P_{m-1}(\mathbf{e}_i)} \)
= \sum_{\tei \subset W} \frac{\PR(\bfe_i=\tei)}{P_{m-1}(\tei)}.
\end{equation}
We will see shortly, and this is crucial to our argument,
that $X_i(\bfW)$ concentrates to 1.  With this in mind, we let
$F_i = F_i(\mathbf{W})$ be the event that 
\be\label{xiw}
|X_i(\bfW)-1| \le \delta^{\frac{1}{3 \times 10^m}}.
\ee 
Very small values of $X_i(W)$, in particular sets $W$ with $X_i(W)=0$,
are problematic for us and must be avoided.  Fortunately, this occurs
with very small probability.

We now define the random variables $\mathbf{e}'_i$ for $i \in I_m$.  
If $F_i(\bfW)$ fails, we set $\bfe_i'=\emptyset$.  Otherwise, if $F_i(\bfW)$
holds, then after conditioning on a fixed value $W$ of $\bfW$,
we choose $\mathbf{e}'_i$ from the essential range of 
$\mathbf{e}_i$ using the conditional probability distribution
\begin{equation}\label{qorn}
 \PR( \mathbf{e}'_i = \tei | \mathbf{W} = W ) := \frac{1_{\tei \subset W}}{X_i(W)}
  \frac{\PR( \mathbf{e}_i = \tei)}{P_{m-1}(\tei)}
\end{equation}
for all $\tei$ in the essential range of $\mathbf{e}_i$, and also require that
the $\mathbf{e}'_i$ are conditionally jointly independent for $i \in
I_m$ on each event $\mathbf{W}=W$.  Note from \eqref{xdef} that
\eqref{qorn} defines a probability distribution, and so the
$\mathbf{e}'_i$ are well defined as random variables.  Informally,
$\mathbf{e}'_i$ is $\mathbf{e}_i$ conditioned to the event
$\mathbf{e}_i \subset W$, and then reweighted by
$P_{m-1}(\mathbf{e}_i)$ to compensate for the bias caused by this
conditioning.

\begin{lem}\label{xconc-quant} We have
\[
 \PR ( F_i(\bfW) ) = 1 - O( \delta^{\frac{1}{3 \times 10^m}} ).
\]
\end{lem}

\begin{proof} By Chebyshev's inequality (Lemma \ref{cheb}), 
it suffices to show that
\begin{equation}\label{mom1}
 \E X_i(\bfW) = 1 + O( \delta^{\frac{1}{10^m}} )
\end{equation}
and
\begin{equation}\label{mom2}
 \E (X_i(\bfW)^2)  = 1 + O( \delta^{\frac{1}{10^m}} ).
\end{equation}
We begin with \eqref{mom1}.  Let $\tei$ be in the essential range of $\bfe_i$.
 From \eqref{p-bound-quant} and \eqref{e2r} we have
$$
\# \tei \le r \le A - 2r(m-1)
$$
and thus by \eqref{xdef} and \eqref{hyp}, we have
\begin{align*}
\E X_i(\bfW) &= \sum_W \PR(\bfW=W) \sum_{\tei\subset W}
  \frac{\PR(\bfe_i=\tei)}{P_{m-1}(\tei)} \\
&= \sum_{\tei} \PR(\bfe_i=\tei) \frac{\PR(\tei \subset \bfW)}{P_{m-1}(\tei)}
=1+ O_{\le}(\delta^{\frac{1}{10^m}}).
\end{align*}

Now we show \eqref{mom2}.  Let $\tei$ and $\hei$
be in the essential range of $\bfe_i$.
From \eqref{p-bound-quant}, \eqref{e2r} we have
$$
\# \tei \cup \hei \leq A - 2r(m-1)
$$
and from \eqref{pje-def} we have
$$ \frac{P_{m-1}( \tei \cup \hei )}{ P_{m-1}(\tei) P_{m-1}(\hei) } 
= \frac{1}{P_{m-1}( \tei \cap \hei )}$$ 
and thus by \eqref{xdef} and \eqref{hyp} we have
\begin{align*}
\E (X_i(\bfW)^2) &= \sum_{\tei,\hei} \PR(\bfe_i=\tei) \PR(\bfe_i=\hei)
\frac{\PR(\tei \cup \hei \subset \bfW)}{P_{m-1}(\tei) P_{m-1}(\hei)}\\
&=\(1+ O_{\le}(\delta^{\frac{1}{10^m}})\) \sum_{\tei,\hei}
\frac{\PR(\bfe_i=\tei) \PR(\bfe_i=\hei)}{P_{m-1}(\tei \cap \hei)}.
\end{align*}
The denominator $P_{m-1}(\tei \cap \hei)$ is 1 if $\tei \cap \hei =
  \emptyset$, and is at least  $\kappa^{r}$ otherwise, 
thanks to \eqref{pjei-size-quant}.
Thus, by \eqref{p-bound-quant}, \eqref{q-form-quant} and a union bound,
\[
\sum_{\tei,\hei}
\frac{\PR(\bfe_i=\tei) \PR(\bfe_i=\hei)}{P_{m-1}(\tei \cap \hei)} = 1 +
O\( \kappa^{-r} \sum_{\tei} \PR(\bfe_i=\tei) \sum_{v\in \tei} \PR(v\in
\bfe_i) \) = 1 + O(r\delta \kappa^{-r}),
\]
and the claim \eqref{mom2} follows from \eqref{delta-small-cor}.
\end{proof}

It remains to verify \eqref{econc-quant}.  
Let $e$ be a fixed subset of $V$ with 
\begin{equation}\label{E-bound-quant}
\# e \leq A - 2rm.
\end{equation}
For any $W$ in the essential range of $\mathbf{W}$, let $Y(W)$ denote the quantity
$$ Y( W ) := \PR\( e \subset W \backslash \bigcup_{i \in I_m} \mathbf{e}'_i | \mathbf{W} = W \).$$
From \eqref{pj-def}, \eqref{pje-def}, \eqref{idem}, our task is now to show that
$$
 \E Y(\mathbf{W})  = \(1 + O_{\leq}( \delta^{1/10^{m+1}} )\)P_{m-1}(e) \exp\Bigg( - \sum_{v \in e} \frac{d_{I_m}(v)}{P_{m-1}(v)} \Bigg).
$$
Clearly $Y(\mathbf{W})$ is only non-zero when $e \subset \mathbf{W}$.  From \eqref{hyp} we have
\begin{equation}\label{ew}
\PR( e \subset \mathbf{W} ) = (1 + O_{\leq}( \delta^{1/10^{m}} ))P_{m-1}(e),
\end{equation}
so it will suffice to show that
$$
 \E( Y( \mathbf{W} ) | e \subset \mathbf{W} ) =
\(1 + O(\delta^{ \frac{1}{9 \times 10^m} } )\)
 \exp\Bigg( - \sum_{v \in e} \frac{d_{I_m}(v)}{P_{m-1}(v)} \Bigg).
$$
From \eqref{epsbound-quant}, \eqref{E-bound-quant}  and
\eqref{delta-small},  we have
$$
 \exp\Bigg( - \sum_{v \in e}\frac{d_{I_m}(v)}{P_{m-1}(v)}  \Bigg) \geq \exp( - A D )
 \ge \delta^{1/10^{m+2}},
$$
so it suffices to show that
\begin{equation}\label{eye-quant}
 \E( Y( \mathbf{W} ) | e \subset \mathbf{W} ) =
\(1 + O(\delta^{ \frac{1}{9 \times 10^m} } )\)
 \exp\Bigg( - \sum_{v \in e}\frac{d_{I_m}(v)}{P_{m-1}(v)}  \Bigg)  + O( \delta^{ \frac{1}{8 \times 10^m} } ).
\end{equation}
Suppose that $W$ is in the essential range of $\mathbf{W}$ with $e
\subset W$.  As the $\mathbf{e}'_i$, $i \in I_m$, are jointly
conditionally independent on the event $\mathbf{W}=W$, we may factor
$Y(W)$ as
$$ Y(W) = \prod_{i \in I_m} (1 - \PR( e \cap \mathbf{e}'_i \neq
\emptyset | \mathbf{W} = W ) ).$$
Since $\bfe_i'=\emptyset$ if $F_i(W)$ fails, we may write 
\[
 Y(W) = \prod_{i \in I_m} (1 - 1_{F_i(W)} \PR( e \cap \mathbf{e}'_i \neq
\emptyset | \mathbf{W} = W ) ).
\]
Now suppose that $i\in I_m$ and that $W$ is such that  $F_i(W)$ holds.
From the union bound we have
$$  \PR( e \cap \mathbf{e}'_i \neq \emptyset | \mathbf{W} = W ) \leq \sum_{v \in e}  \PR( v \in \mathbf{e}'_i | \mathbf{W} = W ).$$
From \eqref{qorn}, \eqref{xiw}, and \eqref{pjei-size-quant}, we have
$$ \PR( v \in \mathbf{e}'_i | \mathbf{W} = W ) =\sum_{\tei:v\in \tei}
\PR(\bfe'_i=\tei | \bfW=W) \ll \kappa^{-r} \PR( v \in \mathbf{e}_i ), $$
and hence by \eqref{q-form-quant}, \eqref{E-bound-quant}
$$ \PR( e \cap \mathbf{e}'_i \neq \emptyset | \mathbf{W} = W ) \ll A \kappa^{-r} \delta / (\# I_m)^{1/2}.$$
From Taylor's expansion, we then have
$$
1 - 1_{F_i(W)}\PR( e \cap \mathbf{e}'_i \neq \emptyset | \bfW = W )  =
\exp \( -  1_{F_i(W)}\PR( e \cap \mathbf{e}'_i \neq \emptyset | \bfW =
W ) + O( (A \kappa^{-r} \delta)^2 / \# I_m ) \).
$$
From \eqref{delta-small-cor}, we have $(A \kappa^{-r} \delta)^2 = O( \delta^{ \frac{1}{9 \times 10^m} } )$, and so
$$ Y(W) = (1 + O(\delta^{ \frac{1}{9 \times 10^m} } )) \exp\left( -1_{F_i(W)}
\sum_{i \in I_m}  \PR( e \cap \mathbf{e}'_i \neq \emptyset |
\mathbf{W} = W ) \right).$$
Next, we apply inclusion-exclusion to write
$$
\PR( e \cap \mathbf{e}'_i \neq \emptyset | \mathbf{W} = W )  = \sum_{v \in e} \PR( v \in \mathbf{e}'_i | \mathbf{W}=W) - O\( \sum_{v,w \in e: v \neq w}
\PR( v, w \in \mathbf{e}'_i | \mathbf{W}=W) \).$$
The error term is handled by summing  \eqref{qorn} over all 
$\tilde{e}_i$ with $v,w\in\tilde{e}_i$, and using
\eqref{xiw} and \eqref{pjei-size-quant}.  For distinct  $v,w \in e$, we have
$$
\PR( v, w \in \mathbf{e}'_i | \mathbf{W}=W )=\sum_{\tilde{e}_i:v,w\in\tilde{e}_i}
\PR(\mathbf{e}'_i=\tilde{e}_i |  \mathbf{W}=W ) \ll \kappa^{-r}
\sum_{\tilde{e}_i:v,w\in\tilde{e}_i} \PR(\mathbf{e}_i=\tilde{e}_i)
 \ll \kappa^{-r} \PR( v, w \in \mathbf{e}_i ).$$
Hence by \eqref{q-form-2-quant}, \eqref{E-bound-quant}
$$ \sum_{i \in I_m} \sum_{\substack{v,w \in e \\ v \neq w}} \PR( v, w \in
\mathbf{e}'_i | \mathbf{W}=W ) \ll \kappa^{-r} A^2 
\max_{\substack{v,w\in e \\ v\ne w}} \; \sum_{i\in I_m} \PR(v,w\in \mathbf{e}_i)
  \ll A^2 \kappa^{-r} \delta.$$
From \eqref{delta-small-cor}, we have $A^2 \kappa^{-r} \delta = O(
\delta^{ \frac{1}{9 \times 10^m} } )$, and so
$$ Y(W) = (1 + O(\delta^{ \frac{1}{9 \times 10^m} } )) \exp\left( -1_{F_i(W)}
\sum_{v \in e} \sum_{i \in I_m}  \PR( v \in \mathbf{e}'_i | \mathbf{W}
= W ) \right).$$

Also we trivially have $0 \leq Y(W) \leq 1$.  Thus, to prove \eqref{eye-quant}, it suffices to show that
$$
\sum_{v \in e} \sum_{i \in I_m} 1_{F_i(\bfW)}\PR( v \in \mathbf{e}'_i | \mathbf{W} ) = 
\sum_{v \in e} \frac{d_{I_m}(v)}{P_{m-1}(v)} + O(\delta^{ \frac{1}{9 \times 10^m} } )
$$
with probability $1 - O( \delta^{ \frac{1}{8 \times 10^m} } )$,
conditionally on the event that $e \subset \mathbf{W}$.  
From \eqref{E-bound-quant}, \eqref{delta-small-cor}, and the union bound,
it thus suffices to show that for each $v \in e$, one has
\be\label{Pvei}
\sum_{i \in I_m}1_{F_i(\bfW)} \PR( v \in \mathbf{e}'_i | \mathbf{W} ) = \frac{d_{I_m}(v)}
{P_{m-1}(v)} + O(\delta^{ \frac{1}{8 \times 10^m} } )
\ee
with probability $1 - O( \delta^{ \frac{1}{7 \times 10^m} } )$,
conditionally on the event that $e \subset \mathbf{W}$.   

We have
\be\label{Pvei2}
1_{F_i(\bfW)}\PR( v \in \mathbf{e}'_i | \mathbf{W} ) = \frac{1_{F_i(\bfW)}}{X_i(\bfW)}
\sum_{\tei:v\in \tei} 1_{\tei \subset \bfW} \frac{\PR(\bfe_i=\tei)}{P_{m-1}(\tei)}
\ee
and, by \eqref{xiw},
\be\label{Xi-decomp}
\frac{1_{F_i(\bfW)}}{X_i(\bfW)} = 1 + O((1-1_{F_i(\bfW)})+\delta^{\frac{1}{3\times 10^m}}).
\ee
Upon inserting \eqref{Pvei2} and \eqref{Xi-decomp} into \eqref{Pvei},
the left side of \eqref{Pvei} breaks into two pieces, a ``main term''
and an ``error term''.

Let us first estimate the error
$$
\sum_{i \in I_m} O\(1-1_{F_i(\bfW)} + \delta^{\frac{1}{3\times 10^m}}\) \sum_{\tei:v\in \tei}
1_{\tei \subset \bfW} \frac{\PR(\bfe_i=\tei)}{P_{m-1}(\tei)}.
$$ 
By \eqref{pjei-size-quant} and \eqref{epsqj-quant}, we may bound this by
$$
O(\kappa^{-r}) \sum_{i \in I_m} (1-1_{F_i(\mathbf{W})}
+\delta^{\frac{1}{3\times 10^m}})) \PR(v \in
\mathbf{e}_i) = O(\kappa^{-r}) d_{I_m}(v)(1-1_{F_i(\mathbf{W})}+ \delta^{\frac{1}{3\times 10^m}}).
$$ 
By Lemma \ref{xconc-quant}, the 
unconditional expectation of this random variable is
$$ O\( \kappa^{-r} \delta^{\frac{1}{3 \times 10^m}} d_{I_m}(v) \).$$
Thus, by \eqref{ew}, the conditional expectation of this random variable to the event $e \subset \mathbf{W}$ is
$$\ll \kappa^{-r} \delta^{\frac{1}{3 \times 10^m}}
\frac{d_{I_m}(v)}{P_{m-1}(e)} \ll \kappa^{-A} \delta^{\frac{1}{3 \times 10^m}}.$$
By \eqref{delta-small-cor}, this can be bounded by
$$ O( \delta^{\frac{2}{7 \times 10^m}}  ).$$
Thus, by Markov's inequality, this error is $O( \delta^{ \frac{1}{7
    \times 10^m} } )$ with probability $1 - O( \delta^{ \frac{1}{7
    \times 10^m} } )$, conditionally on $e \subset \mathbf{W}$.  By
the triangle inequality, it thus suffices to show that the main term satisfies
$$
\sum_{i \in I_m}  \sum_{\tei:v\in \tei}
1_{\tei \subset \bfW} \frac{\PR(\bfe_i=\tei)}{P_{m-1}(\tei)}
= \frac{d_{I_m}(v)}{P_{m-1}(v)} + O(\delta^{ \frac{1}{8 \times 10^m} } )
$$ 
with probability $1 - O( \delta^{ \frac{1}{7 \times 10^m} } )$, conditionally on $e \subset \mathbf{W}$.  

Applying Lemma \ref{cheb} (and \eqref{epsbound-quant}, \eqref{delta-small}), it suffices to show that
\begin{equation}\label{z-1}
\E\( \sum_{i \in I_m} \sum_{\tei:v\in \tei}
1_{\tei \subset \bfW} \frac{\PR(\bfe_i=\tei)}{P_{m-1}(\tei)}
 \Big| e \subset \mathbf{W}\) = 
 \frac{d_{I_m}(v)}{P_{m-1}(v)} + O(\delta^{ \frac{1}{2 \times 10^m} } )
\end{equation}
and
\begin{equation}\label{z-2}
\E\Bigg( \sum_{i,i' \in I_m}  \ssum{\tei:v\in \tei \\ \hei:v\in\hei}
1_{\tei \subset \bfW} \frac{\PR(\bfe_{i}=\tei)}{P_{m-1}(\tei)} 
1_{\hei \subset \bfW} \frac{\PR(\bfe_{i'}=\hei)}{P_{m-1}(\hei)} 
\Big| e \subset \mathbf{W}\Bigg) =  \pfrac{d_{I_m}(v)}{P_{m-1}(v)}^2
 + O(\delta^{ \frac{1}{2 \times 10^m} } ). 
\end{equation}

We begin with \eqref{z-1}.  For any given $i \in I_m$, we have from \eqref{hyp}, \eqref{e2r} that
$$ \frac{\PR (e \cup \tei \subset \bfW )}{\PR(e\subset \bfW)} =
(1 + O( \delta^{1/10^{m}} )) \frac{P_{m-1}(   e \cup \tei )}{P_{m-1}(e)}.$$
By \eqref{pje-def}, we can rewrite
$$ \frac{P_{m-1}(  e\cup \tei )}{P_{m-1}(\tei) P_{m-1}(e)} = \frac{1}{P_{m-1}(v) P_{m-1}(\tei \cap e \backslash \{v\})}.$$
By \eqref{idem}, we may thus write the left-hand side of \eqref{z-1} as
$$
\sum_{i\in I_m} \sum_{\tei: v\in \tei}
\frac{\PR(\bfe_i=\tei)}{P_{m-1}(\tei)}
\frac{\PR(e\cup\tei\subset\bfW)}{\PR(e\subset\bfW)} = 
\frac{1 + O( \delta^{1/10^{m}} )}{P_{m-1}(v)} \sum_{i \in I_m}
\sum_{\tei: v\in \tei} \frac{\PR(\bfe_i=\tei)}{P_{m-1}(\tei \cap e \backslash \{v\})}.
$$
As in the proof of Lemma \ref{xconc-quant}, the denominor is 1
unless $\tei$ and $e \backslash \{v\}$ have a common element, in which case
the denominator is $\ge \kappa^{r}$ by \eqref{pjei-size-quant}.   Thus
$$ \frac{1}{P_{m-1}(\tei \cap e \backslash \{v\})} = 1 + O\Big( \kappa^{-r} 
\sum_{w \in e \backslash \{v\}} 1_{w \in \tei} \Big).$$
From \eqref{epsqj-quant} one has
$$ \sum_{i \in I_m}  \sum_{\tei: v\in \tei} \PR(v\in\bfe_i)  = d_{I_m}(v),$$
and from \eqref{q-form-2-quant} one has
$$ \sum_{i \in I_m}\PR(v,w\in \bfe_i) \leq \delta$$
for all $w \neq v$. Therefore, by \eqref{E-bound-quant}, 
the left side of \eqref{z-1} is
\[
\frac{1 + O( \delta^{1/10^{m}} )}{P_{m-1}(v)} \( d_{I_m}(v) + O(A\delta \kappa^{-r}) \).
\]
The claim now follows from \eqref{delta-small-cor} and \eqref{epsbound-quant}.

Now we prove \eqref{z-2}.  For any $i,i' \in I_m$, we have from \eqref{hyp}, \eqref{e2r} that
$$ \frac{\PR(\tei \cup \hei \cup e \subset \bfW )}{\PR(e\subset \bfW)} =
(1 + O( \delta^{1/10^{m}} )) \frac{P_{m-1}( \tei \cup \hei \cup e)}{P_{m-1}(e)},$$
so we are reduced (after applying \eqref{epsbound-quant}, \eqref{delta-small-cor}) to showing that
$$ \sum_{i,i' \in I_m} \ssum{\tei:v\in \tei \\ \hei:v\in\hei}
\PR(\bfe_{i}=\tei) \PR(\bfe_{i'}=\hei) \frac{P_{m-1}(v)^2 P_{m-1}( \tei \cup
  \hei \cup e)}{P_{m-1}(\tei) P_{m-1}(\hei) P_{m-1}(e)}
 =  d_{I_m}(v)^2 + O(\delta^{ \frac{1}{10^m} } ).$$
The quantity $\frac{P_{m-1}(v)^2 P_{m-1}( \tei \cup \hei \cup e
  )}{P_{m-1}(\tei) P_{m-1}(\hei) P_{m-1}(e)}$ is equal to $1$ when
$\tei,\hei, e$ only intersect at $v$, and is $O(\kappa^{-2r})$ otherwise
thanks to \eqref{pjei-size-quant}.  Hence we may
estimate this ratio by
$$ 1 + O\( \kappa^{-2r} \sum_{w \in e \backslash \{v\}} \(1_{w \in \tei}
+ 1_{w \in \hei}\) \) + O\( \kappa^{-2r} \sum_{w \in \tei \backslash
\{v\}} 1_{w \in \hei}\).$$
From \eqref{epsqj-quant} one has
$$ \sum_{i,i' \in I_m} \PR(v\in\bfe_{i})\PR(v\in\bfe_{i'})  = d_{I_m}(v)^2,$$
so from \eqref{delta-small-cor} it suffices to show that
\begin{align}
\sum_{i,i' \in I_m} \sum_{w \in e \backslash \{v\}} \PR\( v \in
\bfe_{i}, v\in \bfe_{i'}, w \in \bfe_{i} \) &\leq DA \delta, \label{a-1} \\
\sum_{i,i' \in I_m} \sum_{w \in e \backslash \{v\}} \PR\( v \in
\bfe_{i}, v\in \bfe_{i'}, w \in \bfe_{i'} \)
 &\leq DA \delta, \label{a-2} \\
\sum_{i,i' \in I_m} \E \left[  1_{v \in
\bfe_{i}, v\in \bfe_{i'}} \(\# (\bfe_{i} \cap \bfe_{i'}) -1 \) \right]
 &\leq Dr \delta. \label{a-3}
\end{align}

For \eqref{a-1}, we use \eqref{epsqj-quant} to write the left-hand side as
$$ d_{I_m}(v) \sum_{w \in e \backslash \{v\}} \sum_{i \in I_m} \PR(
v, w \in \bfe_{i}),$$
which by \eqref{epsbound-quant}, \eqref{E-bound-quant},
\eqref{q-form-2-quant} is bounded by $D A \delta$, as desired.
Similarly for \eqref{a-2}.  For \eqref{a-3}, we take expectations in
$\bfe_{i'}$ first using \eqref{idem},
\eqref{q-form-2-quant} to upper bound the left-hand side of
\eqref{a-3} by
$$ \sum_{i \in I_m} \E \( 1_{v \in \bfe_{i}} \sum_{w \in
  \bfe_{i} \backslash \{v\}}  \delta\),$$
which by \eqref{p-bound-quant}, \eqref{epsqj-quant},
\eqref{epsbound-quant} is bounded by $Dr\delta$, as desired. 
This proves \eqref{z-2}, which implies \eqref{Pvei} and in turn
\eqref{eye-quant}.
The proof of Theorem \ref{packing-quant} is now complete.

\section{Using a sieve weight}\label{sec:weight}

If $r$ is a natural number, an \emph{admissible $r$-tuple} is a tuple $(h_1,\dots,h_r)$ of distinct integers $h_1,\dots,h_r$ that do not cover all residue
classes modulo $p$, for any prime $p$.  For instance, the tuple $(p_{\pi(r)+1},\dots,p_{\pi(r)+r})$ consisting of the first $r$ primes larger than $r$ is an admissible $r$-tuple.

We will establish Theorem \ref{sieve-primes-2} by a probabilistic argument involving a certain weight function, the details of which may be found in the following. 

\begin{thm}[Existence of good sieve weight]\label{weight}
Let $x$ be a sufficiently large real number and let $y$ be any quantity obeying \eqref{ydef}.  Let $\PP, \QQ$ be defined by \eqref{p-def}, \eqref{q-def}.  Let $r$ be a positive integer with
\begin{equation}\label{r-bound}
r_0 \leq r \leq \log^{1/5} x
\end{equation}
for some sufficiently large absolute constant $r_0$, and let $(h_1,\dots,h_r)$ be an admissible $r$-tuple contained in $[2r^2]$.  Then one can find a positive quantity
\begin{equation}\label{alpha-crude}
\tau \geq x^{-o(1)}
\end{equation}
and a positive quantity $u = u(r)$ depending only on $r$ with
\begin{equation}\label{u-bound}
u \asymp \log r
\end{equation}
and a non-negative function $w: \PP \times \Z \to \R^+$ supported on $\PP \times (\Z \cap [-y,y])$ with the following properties:
\begin{itemize}
\item Uniformly for every $p \in \PP$, one has
\begin{equation}\label{wap}
 \sum_{n \in \Z} w(p,n) = \(1 + O\( \frac{1}{\log^{10}_2 x} \)\) 
\tau \frac{y}{\log^r x}.
\end{equation}
\item Uniformly for every $q \in \QQ$ and $i=1,\dots,r$, one has
\begin{equation}\label{wbp}
 \sum_{p \in \PP} w( p, q - h_i p ) = \(1 + O\( \frac{1}{\log^{10}_2 x} \)\) \tau \frac{u}{r} \frac{x}{2 \log^r x}.
\end{equation}
\item Uniformly for every $h = O(y/x)$ that is not equal to any of the $h_i$, one has
\begin{equation}\label{wcp}
\sum_{q \in \QQ} \sum_{p \in \PP} w( p, q - h p ) = O\( \frac{1}{\log^{10}_2 x} \tau \frac{x}{\log^r x} \frac{y}{\log x}\).
\end{equation}
\item Uniformly for all $p \in \PP$ and $n \in \Z$,
\begin{equation}\label{w-triv}
w(p,n) = O( x^{1/3+o(1)} ).
\end{equation}
\end{itemize}
\end{thm}

\begin{remark} One should think of $w(p,n)$ as being a smoothed out indicator function for the event that $n+h_1p,\dots,n+h_r p$ are all almost primes in $[y]$.  As essentially discovered in \cite{maynard-gaps}, by choosing the smoothing correctly, one can ensure that approximately $\log r$ of the elements of this tuple $n+h_1p,\dots,n+h_r p$  are genuinely prime rather than almost prime, when weighted by $w(p,n)$; this explains the presence of the bounds \eqref{u-bound}.  The estimate \eqref{wcp} is not, strictly speaking, needed for our current argument; however, it is easily obtained by our methods, and will be of use in a followup work \cite{FMT} to this paper in which the analogue of Theorem \ref{mainthm} for chains of large gaps is established.
\end{remark}

The proof of this theorem will rely on the estimates for multidimensional prime-detecting sieves established by the fourth author in \cite{maynard-dense}, and will be the focus of subsequent sections.  In this section, we show how Theorem \ref{weight} implies Theorem \ref{sieve-primes-2}.

Let $x, c, y, z, \cS, \PP, \QQ$ be as in Theorem \ref{sieve-primes-2}.
We set $r$ to be the maximum value permitted by Theorem \ref{weight}, namely
\begin{equation}\label{r-def}
 r := \lfloor \log^{1/5} x \rfloor 
\end{equation}
and let $(h_1,\dots,h_r)$ be the admissible $r$-tuple consisting of the first $r$ primes larger than $r$, thus $h_i = p_{\pi(r)+i}$ for $i=1,\dots,r$.  From the prime number theorem we have $h_i = O( r \log r )$ for $i=1,\dots,r$, and so we have $h_i \in [2r^2]$ for $i=1,\dots,r$ if $x$ is large enough (there are
many other choices possible, e.g. $(h_1,\ldots,h_r)=(1^2,3^2,\ldots,(2r-1)^2)$).  We now invoke Theorem \ref{weight} to obtain quantities $\tau,u$ and a weight $w: \PP \times \Z \to \R^+$ with the stated properties.

For each $p \in \PP$, let $\tilde{\mathbf{n}}_p$ denote the random integer with probability density
$$ \PR( \tilde{\mathbf{n}}_p = n ) := \frac{w(p,n)}{  \sum_{n' \in \Z} w(p,n') }$$
for all $n \in \Z$ (we will not need to impose any independence conditions on the $\tilde{\mathbf{n}}_p$).  From \eqref{wap}, \eqref{wbp} we have
\begin{equation}\label{wbp-diff}
\sum_{p \in \PP} \PR( q = \tilde{\mathbf{n}}_p + h_ip ) = \(1 + O\( \frac{1}{\log^{10}_2 x} \)\) \frac{u}{r} \frac{x}{2y} \qquad (q\in \QQ,1\le i\le r).
\end{equation}
Also, from \eqref{wap}, \eqref{w-triv}, \eqref{alpha-crude} one has
\begin{equation}\label{w-triv-diff}
 \PR( \tilde{\mathbf{n}}_p = n ) \ll x^{-1/2 - 1/6 +o(1)}
\end{equation}
for all $p \in \PP$ and $n \in \Z$.

We choose the random vector
$\vec{\mathbf{a}} := (\mathbf{a}_s \mod s)_{s \in \cS}$ by selecting each $\mathbf{a}_s \mod s$ uniformly at random from $\Z/s\Z$, independently in $s$ and independently of the $\tilde{\mathbf{n}}_p$.
The resulting sifted set $S(\vec{\mathbf{a}})$ is a random periodic subset of $\Z$ with density
$$ \sigma := \prod_{s \in \cS} \(1 - \frac{1}{s}\).$$
From the prime number theorem (with sufficiently strong error term), 
\eqref{zdef} and \eqref{s-def},
\[
\sigma = \(1 + O\(\frac{1}{\log^{10}_2 x}\)\) \frac{\log(\log^{20} x)}{\log z} \\
= \(1 + O\(\frac{1}{\log^{10}_2 x}\)\) \frac{80 \log_2 x}{\log x \log_3 x / \log_2 x},
\]
so in particular we see from \eqref{ydef} that
\begin{equation}\label{gamma-y}
\sigma y = \(1 + O\(\frac{1}{\log^{10}_2 x}\)\) 80 c x \log_2 x.
\end{equation}
We also see from \eqref{r-def} that
\begin{equation}\label{gamma-small}
\sigma^r = x^{o(1)}.
\end{equation}

We have a useful correlation bound:

\begin{lem}\label{gamma-cor}  Let $t \le \log x$ be a natural number, and let $n_1,\dots,n_t$ be distinct integers of magnitude $O(x^{O(1)})$.  Then one has
$$ \PR( n_1,\dots,n_t \in S(\vec{\mathbf{a}}) ) = \(1 + O\pfrac{1}{\log^{16} x}\) \sigma^t.$$
\end{lem}

\begin{proof}  For each $s \in \cS$, the integers $n_1,\dots,n_t$ occupy $t$ distinct residue classes modulo $s$, unless $s$ divides one of $n_i-n_j$ for $1 \leq i < j \leq t$.  Since $s \geq \log^{20} x$ and the $n_i-n_j$ are of size $O( x^{O(1)} )$, the latter possibility occurs at most $O( t^2 \log x ) = O( \log^3 x )$ times. Thus the probability that $\mathbf{a}_s \mod s$ avoids all of the $n_1,\dots,n_t$ is equal to $1 - \frac{t}{s}$ except for $O(\log^3 x)$ values of $s$, where it is instead $(1 + O( \frac{1}{\log^{19} x} )) (1-\frac{t}{s})$.  Thus,
\begin{align*}
\PR( n_1,\dots,n_t \in S(\vec{\mathbf{a}}) )&= \(1 + O\pfrac{1}{\log^{19} x} \)^{O( \log^3 x)} \prod_{s \in \cS} \(1 - \pfrac{t}{s} \)  \\
&= \(1 + O\pfrac{1}{\log^{16} x}\) \sigma^t \prod_{s \in \cS} \(1 + O \pfrac{t^2}{s^2} \) \\
&= \(1 + O\pfrac{1}{\log^{16} x}\) \sigma^t. \qedhere
\end{align*}
\end{proof}

Among other things, this gives the claim \eqref{treat}:
 
\begin{cor}\label{s0}  With probability $1-o(1)$, we have
\begin{equation}\label{qqa}
\# (\QQ \cap S(\vec{\mathbf{a}}) ) \sim \sigma \frac{y}{\log x} \sim
 80 c \frac{x}{\log x} \log_2 x.
\end{equation}
\end{cor}

\begin{proof} 
From Lemma \ref{gamma-cor}, we have
$$ \E \# (\QQ \cap S(\vec{\mathbf{a}}))  = \(1 + O\pfrac{1}{\log^{16} x}\) \sigma \# \QQ$$
and
$$ \E \# \big( (\QQ \cap S(\vec{\mathbf{a}})) \big)^2  = \(1 + O\pfrac{1}{\log^{16} x}\) ( \sigma \# \QQ + \sigma^2 (\# \QQ) (\# \QQ-1)),$$
and so by the prime number theorem we see that the random variable $\# \QQ \cap S(\vec{\mathbf{a}})$ has mean $(1 + o(\frac{1}{\log_2 x})) \sigma \frac{y}{\log x}$ and variance $O\( \frac{1}{\log^{16} x} (\sigma \frac{y}{\log x})^2 \)$.  The claim then follows from Chebyshev's inequality (with plenty of room to spare).
\end{proof}

For each $p \in \PP$, we consider the quantity
\begin{equation}\label{xp-def}
 X_p(\vec{a}) := \PR( \tilde{\mathbf{n}}_p + h_ip \in S(\vec{a}) \text{ for all } i=1,\dots,r ),
\end{equation}
and let $\PP(\vec a)$ denote the set of all the primes $p \in \PP$ such that
\begin{equation}\label{sumn}
X_p(\vec a) = \(1 + O_{\leq}\(\frac{1}{\log^3 x}\)\) \sigma^r.
\end{equation}

In light of Lemma \ref{gamma-cor}, we expect most primes
in $\PP$ to lie in $\PP(\vec{a})$, and this will be confirmed
below (Lemma \ref{smc}).
We now define the random variables $\mathbf{n}_p$ as follows.  Suppose we are in the event $\vec{\mathbf{a}} = \vec a$ for some $\vec a$ in the range of $\vec{\mathbf{a}}$.  If $p \in \PP \backslash \PP(\vec{a})$, we set $\mathbf{n}_p=0$.  Otherwise, if $p \in \PP(\vec{a})$, we define $\mathbf{n}_p$ to be the random integer with conditional probability distribution
\begin{equation}\label{xpa}
 \PR( \mathbf{n}_p = n | \vec{\mathbf{a}} = \vec a ) := \frac{Z_p(\vec{a};n)}
{X_p(\vec a)}, \quad Z_p(\vec{a};n) = 1_{n+h_jp\in
S(\vec{a})\text{ for }j=1,\ldots,r} \PR(\tilde{\mathbf{n}}_p=n),
\end{equation}
with the $\mathbf{n}_p$ ($p\in\PP(\vec{a})$) jointly independent, conditionally on the event $\vec{\mathbf{a}} = \vec a$.  From \eqref{xp-def} we see that these random variables are well defined.


\begin{lem}\label{smc-2}  
With probability $1-o(1)$, we have
\begin{equation}\label{sumno}
\sigma^{-r} \sum_{i=1}^r \sum_{p\in \PP(\vec{\ba})} 
Z_p(\vec{\ba};q-h_ip) = \(1 + O\(\frac{1}{\log^{3}_2 x}\)\) 
\frac{u}{\sigma} \frac{x}{2y}
\end{equation}
for all but at most $\frac{x}{2\log x \log_2 x}$ of the primes $q \in \QQ \cap S(\vec{\mathbf{a}})$.
\end{lem}

Let $\vec{a}$ be good and $q\in \QQ \cap S(\vec{a})$.  Substituting 
 definition \eqref{xpa} into the left hand side of of \eqref{sumno},
using \eqref{sumn}, and observing that $q=\mathbf{n}_p+h_ip$ is only 
possible if $p\in \PP(\vec{\ba})$, we find that
\begin{align*}
\sigma^{-r} \sum_{i=1}^r \sum_{p\in \PP(\vec{a})} 
Z_p(\vec{a};q-h_ip) &= \sigma^{-r}  \sum_{i=1}^r \sum_{p\in
  \PP(\vec{a})} X_p(\vec{a}) \PR(\mathbf{n}_p=q-h_ip |
\vec{\ba}=\vec{a})\\
&=\(1+O\pfrac{1}{\log^3 x}\)  \sum_{i=1}^r \sum_{p\in \PP(\vec{a})}
 \PR(\mathbf{n}_p=q-h_ip | \vec{\ba}=\vec{a})\\
&= \(1+O\pfrac{1}{\log^3 x}\) \sum_{p\in\PP} \PR(q\in
\mathbf{e}_p(\vec{a}) |  \vec{\ba}=\vec{a}),
\end{align*}
where $\mathbf{e}_p(\vec{a}) = \{ \mathbf{n}_p+h_ip:1\le i\le r \}
\cap \QQ \cap S(\vec{a})$ is as defined in Theorem
\ref{sieve-primes-2}.
Relation \eqref{good} (that is, $\vec{\ba}$ is good with probability
$1-o(1)$) follows
upon noting that by \eqref{r-def}, \eqref{u-bound} and \eqref{gamma-y},
$$
C := \frac{u}{\sigma} \ \frac{x}{2y} \asym \frac{1}{c}.
$$

Before proving Lemma \ref{smc-2}, 
we first confirm that $\PP \backslash \PP(\vec{\ba})$ is small with high
probability.

\begin{lem}\label{smc}   
With probability $1-O(1/\log^3 x)$, $\PP(\vec{\mathbf{a}})$ contains
 all but $O( \frac{1}{\log^3 x} \frac{x}{\log x})$ of the primes $p \in \PP$. 
In particular, $\E \# \PP(\vec{\mathbf{a}}) = \# \PP (1+O(1/\log^3 x))$.
\end{lem}

\begin{proof}  By linearity of expectation and Markov's inequality, it suffices to show that for each $p \in \PP$, we have $p\in\PP(\vec{\mathbf{a}})$
  with probability $1 - O(\frac{1}{\log^6 x})$.  By Lemma \ref{cheb}, it suffices to show that
\begin{equation}\label{load-1}
\E X_p(\vec{\mathbf{a}}) =
\PR( \tilde{\mathbf{n}}_p + h_ip \in S(\vec{\mathbf{a}}) \text{ for all } i=1,\dots,r ) = \(1 + O\pfrac{1}{\log^{12} x}\) \sigma^r
\end{equation}
and
\begin{equation}\label{load-2}
\E X_p(\vec{\mathbf{a}})^2  =
\PR( \tilde{\mathbf{n}}^{(1)}_p + h_ip, \tilde{\mathbf{n}}^{(2)}_p + h_ip \in S(\vec{\mathbf{a}}) \text{ for all } i=1,\dots,r ) = \(1 + O \pfrac{1}{\log^{12} x}\) \sigma^{2r},
\end{equation}
where $\tilde{\mathbf{n}}^{(1)}_p, \tilde{\mathbf{n}}^{(2)}_p$ are independent copies of $\tilde{\mathbf{n}}_p$ that are also independent of $\vec{\mathbf{a}}$.  

The claim \eqref{load-1} follows from Lemma \ref{gamma-cor}  (performing the conditional expectation over $\tilde{\mathbf{n}}_p$ first).  A similar application of Lemma \ref{gamma-cor} allows one to write the left-hand side of \eqref{load-2} as
$$
\(1 + O\pfrac{1}{\log^{16} x}\) \E \sigma^{\# \{ \tilde{\mathbf{n}}^{(l)}_p + h_ip: i=1,\dots,r; l=1,2\}}.$$
From \eqref{w-triv-diff} we see that the quantity $\# \{ \tilde{\mathbf{n}}^{(l)}_p + h_ip: i=1,\dots,r; l=1,2\}$ is equal to $2r$ with probability $1 - O(x^{-1/2-1/6+o(1)})$, and is less than $2r$ otherwise.  The claim now follows from \eqref{gamma-small}.
\end{proof}

\begin{proof}[Proof of Lemma \ref{smc-2}]
We first show that replacing $\PP(\vec{\ba})$ with $\PP$
has negligible effect on the sum, with probability $1-o(1)$.  
Fix $i$ and susbtitute $n=q-h_ip$.  By Markov's inequality, it
suffices to show that
\be\label{soo-2}
\E \sum_n \sigma^{-r} \sum_{p\in \PP\backslash \PP(\vec{\ba})} Z_p(\vec{\ba};n)
=o\(
\frac{u}{\sigma} \frac{x}{2y}\ \frac{1}{r} \frac{1}{\log_2^3 x}\ \frac{x}{\log x\log_2 x} \).
\ee
By Lemma \ref{gamma-cor}, we have
\begin{align*}
\E\ \sum_{n} \sigma^{-r} \sum_{p \in \PP} Z_p(\vec{\ba};n)
&= \sigma^{-r} \sum_{p \in \PP} \sum_{n}  \PR(\tilde{\mathbf{n}}_p=n)
\PR( n+h_jp\in S(\vec{\ba})\text{ for }j=1,\ldots,r ) \\
&=\(1+O\pfrac{1}{\log^{16} x}\) \# \PP.
\end{align*}
Next, by \eqref{sumn} and Lemma \ref{smc} we have
\begin{align*}
\E\ \sum_{n} \sigma^{-r} &\sum_{p \in \PP(\vec{\ba})}  Z_p(\vec{\ba};n)
= \sigma^{-r} \sum_{\vec{a}} \PR(\vec{\ba}=\vec{a}) \sum_{p\in \PP(\vec{a})} 
X_p(\vec{a}) \\
&=\(1 + O\pfrac{1}{\log^3 x}\) \ 
\E \; \# \PP(\vec{\mathbf{a}}) 
=\(1 + O\pfrac{1}{\log^3 x}\) \# \PP ;
\end{align*}
subtracting, we conclude that the left-hand side of \eqref{soo-2} is
$O(\#\PP/\log^3 x)=O(x/\log^4 x)$. 
The claim then follows from \eqref{ydef} and \eqref{r-bound}.

By \eqref{soo-2}, it suffices to show that with probability $1-o(1)$,
for all but at most $\frac{x}{2\log x\log_2 x}$ primes $q \in \QQ \cap S(\vec{\mathbf{a}})$, one has
\begin{equation}\label{soo}
\sum_{i=1}^r \sum_{p\in \PP} Z_p(\vec{\ba};q-h_ip)
 = \(1 + O_{\le}\(\frac{1}{\log^{3}_2 x}\)\) 
\sigma^{r-1} u \frac{x}{2y}.
\end{equation}

Call a prime $q \in \QQ$ \emph{bad} if $q \in \QQ \cap S(\vec{\mathbf{a}})$ but
\eqref{soo} fails.  
Using Lemma \ref{gamma-cor} and \eqref{wbp-diff}, we have
\begin{align*}
\E \bigg[\sum_{q\in\QQ\cap S(\vec{\mathbf{a}}) } \sum_{i=1}^r \sum_{p \in \PP} Z_p(\vec{\ba};q-h_ip)
\bigg]&=\sum_{q,i,p} \PR( q + (h_j-h_i) p \in S(\vec{\mathbf{a}}) \text{ for all } j=1,\dots,r)\PR(\tilde{\mathbf{n}}_p=q-h_ip)\\
&=\(1+O\pfrac{1}{\log_2^{10} x}\) \frac{\sigma y}{\log x} \ 
\sigma^{r-1} u \ \frac{x}{2y}
\end{align*}
and
\begin{align*}
\E \bigg[ \sum_{q\in\QQ\cap S(\vec{\mathbf{a}}) } \bigg( \sum_{i=1}^r \sum_{p\in\PP} Z_p (\vec{\ba};q-h_ip)\bigg)^2
\bigg]&=\sum_{\substack{p_1,p_2,q \\ i_1,i_2}} \PR( q + (h_j-h_{i_\ell}) p_\ell \in S(\vec{\mathbf{a}}) 
\text{ for } j=1,\dots,r;\ell=1,2) \\
& \qquad \times \PR(\tilde{\mathbf{n}}^{(1)}_{p_1}=q-h_{i_1}p_1) 
\PR(\tilde{\mathbf{n}}^{(2)}_{p_2}=q-h_{i_2}p_2) \\
&=\(1+O\pfrac{1}{\log_2^{10} x}\)  \frac{\sigma y}{\log x} \ 
\(\sigma^{r-1} u \ \frac{x}{2y}\)^2,
\end{align*}
where $(\tilde{\mathbf{n}}^{(1)}_{p_1})_{p_1 \in \PP}$ and $(\tilde{\mathbf{n}}^{(2)}_{p_2})_{p_2 \in \PP}$ are independent copies of $(\tilde{\mathbf{n}}_p)_{p \in \PP}$ over $\vec{\mathbf{a}}$.  In the last step we used the fact that the terms with
$p_1=p_2$ contribute negligibly.

By Chebyshev's inequality (Lemma \ref{cheb}) it follows that the number of
bad $q$ is $\ll \frac{\sigma y}{\log x} \frac{1}{\log_2^3 x} \ll \frac{x}{\log x\log_2^2 x}$ with probability $1-O(1/\log_2 x)$.  This concludes the proof.
\end{proof}

It remains to establish Theorem \ref{weight}.  This is the objective of the remaining sections of the paper.

\section{Multidimensional sieve estimates}\label{sec:sievemulti}

We now recall a technical multidimensional sieve estimate from \cite{maynard-dense} (a minor variant of \cite[Proposition 6.1]{maynard-dense}).  In this section we will follow the notation from \cite{maynard-dense}, which is a little different from that in the rest of this paper, with the exception that we will take the set denoted $\mathcal{P}$ in that paper to be equal to the set $\mathscr{P}$ of all primes from the outset.

A \emph{linear form} will be a function $L: \Z \to \Z$ of the form $L(n) = l_1 n + l_2$ with integer coefficients $l_1,l_2$ and $l_1 \neq 0$. Let ${\mathcal A}$ be a set of integers.  Given a linear form $L(n) = l_1 n + l_2$, we define the sets
\begin{align*}
{\mathcal A}(x) &:= \{ n \in {\mathcal A}: x \leq n \leq 2x\}, \\
{\mathcal A}(x;q,a) &:= \{ n \in {\mathcal A}(x): n \equiv a\ \pmod q\}, \\
{\mathscr P}_{L,{\mathcal A}}(x) &:= L({\mathcal A}(x)) \cap {\mathscr P}, \\
{\mathscr P}_{L,{\mathcal A}}(x; q,a) &:= L({\mathcal A}(x;q,a)) \cap {\mathscr P},
\end{align*}
for any $x > 0$ and congruence class $a \mod q$, and define the quantity
$$ \varphi_L(q) := \varphi(|l_1| q) / \varphi(|l_1|),$$
where $\varphi$ is the Euler totient function.  We recall the standard bounds
\begin{equation}\label{xlog}
 X \geq \varphi(X) \gg \frac{X}{\log_2 X}
\end{equation}
since $\varphi(X)/X$ is smallest when $X$ is composed only of primes $\ll \log{X}$.  Thanks to this bound, most factors of the form $\frac{X}{\varphi(X)}$ appearing below become relatively harmless, and we recommend that they may be ignored for a first reading.

A finite set ${\mathcal L} = \{ L_1,\dots,L_k\}$ of linear forms is said to be \emph{admissible} if $\prod_{i=1}^k L_i(n)$ has no fixed prime divisor; that is, for every prime $p$ there exists an integer $n_p$ such that $\prod_{i=1}^k L_i(n_p)$ is not divisible by $p$.

\begin{definition}\cite{maynard-dense}  Let $x$ be a large quantity, let ${\mathcal A}$ be a set of integers, ${\mathcal L} = \{L_1,\dots,L_k\}$ a finite set of linear forms, and $B$ a natural number. We allow ${\mathcal A}, {\mathcal L}, k, B$ to vary with $x$.  Let $0 < \theta < 1$ be a quantity independent of $x$.  Let ${\mathcal L}'$ be a subset of ${\mathcal L}$.  We say that the tuple $({\mathcal A}, {\mathcal L}, {\mathscr P}, B, x, \theta)$ \emph{obeys Hypothesis 1 at ${\mathcal L}'$} if we have the following three estimates:
\begin{enumerate}
\item[(1)] (${\mathcal A}(x)$ is well-distributed in arithmetic progressions) We have
$$ \sum_{q \leq x^\theta} \max_a \left| \# {\mathcal A}(x;q,a) - \frac{\# {\mathcal A}(x)}{q} \right| \ll \frac{\# {\mathcal A}(x) }{\log^{100k^2} x}.$$
\item[(2)] (${\mathscr P}_{L,{\mathcal A}}(x)$ is well-distributed in arithmetic progressions) For any $L \in {\mathcal L}'$, we have
$$ \sum_{q \leq x^\theta;\,(q,B)=1} \max_{a: (L(a),q)=1} \left| \# {\mathscr P}_{L,{\mathcal A}}(x;q,a) - \frac{\# {\mathscr P}_{L,{\mathcal A}}(x)}{\varphi_L(q)} \right| \ll \frac{\# {\mathscr P}_{L,{\mathcal A}}(x) }{\log^{100k^2} x}.$$
\item[(3)] (${\mathcal A}(x)$ not too concentrated)  For any $q < x^\theta$ and $a \in \Z$ we have
$$ \# {\mathcal A}(x;q,a)  \ll \frac{\# {\mathcal A}(x)}{q}.$$
\end{enumerate}
\end{definition}

In \cite{maynard-dense} this definition was only given in the case ${\mathcal L}'={\mathcal L}$, but we will need the (mild) generalization to the case in which ${\mathcal L}'$ is a (possibly empty) subset of ${\mathcal L}$.

As is common in analytic number theory, we will have to address the possibility of an exceptional Siegel zero.  As we want to keep all our estimates effective, we will not rely on Siegel's theorem or its consequences (such as the Bombieri-Vinogradov theorem).  Instead, we will rely on the Landau-Page theorem, which we now recall.  Throughout, $\chi$ denotes a Dirichlet character.

\begin{lem}[Landau-Page theorem]\label{page}  Let $Q \ge 100$.  Suppose that $L(s,\chi) = 0$ for some primitive character $\chi$ of modulus at most $Q$, and some $s = \sigma + it$.  Then either
$$ 1-\sigma \gg \frac{1}{\log(Q (1+|t|))},$$
or else $t=0$ and $\chi$ is a quadratic character $\chi_Q$, which is unique.  Furthermore, if $\chi_Q$ exists, then its conductor $q_Q$ is square-free apart from a factor of at most $4$, and obeys the lower bound
$$ q_Q \gg \frac{\log^2 Q}{\log^2_2{Q}}.$$
\end{lem}

\begin{proof}  See e.g. \cite[Chapter 14]{Da}.  The final estimate follows from the bound $1-\beta\gg q^{-1/2}\log^{-2}{q}$ for a real zero $\beta$ of $L(s,\chi)$ with $\chi$ of modulus $q$, which can also be found in \cite[Chapter 14]{Da}. 
\end{proof}

We can then eliminate the exceptional character by deleting at most one prime
factor of $q_Q$.

\begin{cor}\label{page-cor}  Let $Q \ge 100$.  Then there exists a quantity $B_Q$ which is either equal to $1$ or is a prime of size
$$ B_Q \gg \log_2 Q$$
with the property that
$$ 1-\sigma \gg  \frac{1}{\log(Q (1+|t|))}$$
whenever $L(\sigma+it,\chi)=0$ and $\chi$ is a character of modulus at most $Q$ and coprime to $B_Q$.
\end{cor}

\begin{proof}  If the exceptional character $\chi_Q$ from Lemma \ref{page} does not exist, then take $B_Q := 1$; otherwise we take $B_Q$ to be the largest prime factor of $q_Q$.  As $q_Q$ is square-free apart from a factor of at most $4$, we have $\log q_Q \ll B_Q$ by the prime number theorem, and the claim follows.
\end{proof}

We will only need the above definition in the following special case:

\begin{lem}\label{landau-page}  Let $x$ be a large quantity.  Then there exists a natural number $B \le x$, which is either $1$ or a prime, such that the following holds.  Let ${\mathcal A} := \Z$, let $\theta := 1/3$, and let ${\mathcal L} = \{ L_1,\dots,L_k\}$ be a finite set of linear forms $L_i(n) = a_i n + b_i$ (which may depend on $x$) with $k \leq \log^{1/5} x$, $1 \leq |a_i| \le \log x$, and $|b_i| \le x\log^{2}{x}$.  Let $x\le y\le x\log^{2}{x}$, and let ${\mathcal L}'$ be a subset of ${\mathcal L}$ such that $L_i$ is non-negative on $[y,2y]$ and $a_i$ is coprime to $B$ for all $L_i \in{\mathcal L}'$.  Then $({\mathcal A}, {\mathcal L}, {\mathscr P}, B, y, \theta)$ obeys Hypothesis 1 at ${\mathcal L}'$ with absolute implied constants (i.e. the bounds in Hypothesis 1 are uniform over all such choices of $\mathcal{L}$ and $y$).
\end{lem}

\begin{proof}
Parts (1) and (3) of Hypothesis 1 are easy; the only difficult verification is (2).  We apply Corollary \ref{page-cor} with $Q := \exp( c_1 \sqrt{\log x} )$ for some small absolute constant $c_1$ to obtain a quantity $B := B_Q$ with the stated properties.  By the Landau-Page theorem (see \cite[Chapter 20]{Da}), we have that if $c_1$ is sufficiently small then we have the effective bound
\begin{equation}
\phi(q)^{-1}\sideset{}{^*}\sum_{\chi}|\psi(z,\chi)|\ll x\exp(-3c\sqrt{\log{x}})
\end{equation}
for all $1<q<\exp(2c\sqrt{\log{x}})$ with $(q,B)=1$ and all $z\le x\log^4{x}$. Here the summation is over all primitive $\chi\mod{q}$ and $\psi(z,\chi)=\sum_{n\le z}\chi(n)\Lambda(n)$. Following a standard proof of the Bombieri-Vinogradov Theorem (see \cite[Chapter 28]{Da}, for example), we have (for a suitable constant $c>0$)
\begin{equation}
\sum_{\substack{q<x^{1/2-\epsilon}\\ (q,B)=1}}\sup_{\substack{(a,q)=1\\ z\le x\log^4{x}}}\Bigl|\pi(z;q,a)-\frac{\pi(z)}{\phi(q)}\Bigr|\ll x\exp(-c\sqrt{\log{x}})+\log{x}\sum_{\substack{q<\exp(2c\sqrt{\log{x}})\\ (q,B)=1}}\sideset{}{^*}\sum_\chi\sup_{z\le x\log^4{x}}\frac{|\psi(z,\chi)|}{\phi(q)}.
\end{equation}
Combining these two statements and using the triangle inequality gives the bound required for (2).
\end{proof}
We now recall the construction of sieve weights from \cite[Section 7]{maynard-dense}. On first reading we recommend the reader not pay too much attention to the details; the key point is the existence of a weight $w(n)$ which will establish Theorem \ref{weight}. The reason it is necessary to know the construction is the technical issue that the weights $w(n)$ depend on a given admissible set of linear forms, and we require that the final estimates obtained are essentially uniform over similar admissible sets.

Let $W := \prod_{p \leq 2k^2;\, p \nmid B} p$.  For each prime $p$ not dividing $B$, 
let $r_{p,1}({\mathcal L}) < \dots < r_{p,\omega_{\mathcal L}(p)}({\mathcal L})$ be the elements $n$ of $[p]$ for which $p| \prod_{i=1}^k L_i(n)$. If $p$ is also coprime to $W$, then for each $1 \leq a \leq \omega_{\mathcal L}(p)$, let $j_{p,a} = j_{p,a}({\mathcal L})$ denote the least element of $[k]$ such that $p | L_{j_{p,a}}( r_{p,a}({\mathcal L}) )$.  

Let ${\mathcal D}_k(\mathcal{L})$ denote the set
\[\begin{split} {\mathcal D}_k(\mathcal{L}) := \{ (d_1,\dots,d_k) \in \mathbf{N}^k: \mu^2(d_1 \dots d_k)=1; (d_1 \dots d_k,WB)=1;\\
 (d_j,p)=1 \text{ whenever } p \nmid BW \text{ and } j \neq j_{p,1},\dots,j_{p,\omega_{\mathcal L}(p)} \}.
\end{split}\]

Define the singular series
$$ {\mathfrak S}({\mathcal L}) := \prod_{p \nmid B} \(1 - \frac{\omega_{\mathcal L}(p)}{p}\) \(1 - \frac{1}{p}\)^{-k},$$
and
$$ {\mathfrak S}_{WB}({\mathcal L}) := \prod_{p \nmid WB} \(1 - \frac{\omega_{\mathcal L}(p)}{p}\) \(1 - \frac{1}{p}\)^{-k},$$
the function
$$ \varphi_{\omega_{\mathcal L}}(d) := \prod_{p|d} (p - \omega_{\mathcal L}(p)),$$
and let $R$ be a quantity of size
$$ x^{\theta/10} \leq R \leq x^{\theta/3}.$$
Let $F:\mathscr{R}^k\rightarrow\mathscr{R}$ be a smooth function supported on the simplex
$$\mathcal{R}_k=\{ (t_1,\dots,t_k) \in \R_+^k: t_1+\dots+t_k \leq 1 \}.$$
For any $(r_1,\dots,r_k) \in {\mathcal D}_k(\mathcal{L})$ define
$$ y_{(r_1,\dots,r_k)}({\mathcal L}) := \frac{1_{{\mathcal D}_k(\mathcal{L})}(r_1,\dots,r_k) W^k B^k}{\varphi(WB)^k} {\mathfrak S}_{WB}({\mathcal L}) F\( \frac{\log r_1}{\log R}, \dots, \frac{\log r_k}{\log R} \).$$
For any $(d_1,\dots,d_k) \in {\mathcal D}_k(\mathcal{L})$, define
$$ \lambda_{(d_1,\dots,d_k)}({\mathcal L}) := \mu(d_1 \dots d_k) d_1 \dots d_k \sum_{d_i | r_i \text{ for } i=1,\dots,k} \frac{y_{(r_1,\dots,r_k)}({\mathcal L})}{\varphi_{\omega_{\mathcal L}}(r_1 \dots r_k)}, $$
and then define the function $w =  w_{k, {\mathcal L}, B, R}: \Z \to \R^+$ by
\be\label{wdef}
 w(n) := \left( \sum_{d_1,\dots,d_k: d_i | L_i(n) \text{ for all } i} \lambda_{(d_1,\dots,d_k)}(\mathcal{L}) \right)^2.
\ee
We note that the restriction of the support of $F$ to $\mathcal{R}_k$ means that $\lambda_{(d_1,\dots,d_k)}(\mathcal{L})$ and $y_{(r_1,\dots,r_k)}$ are supported on the set
$$ \mathcal{S}_k(\mathcal{L})=\mathcal{D}_k(\mathcal{L})\cap\{(d_1,\dots,d_k):\prod_{i=1}^kd_i\le R\}. $$
We then have the following result, a slightly modified form of Proposition 6.1 from \cite{maynard-dense}:

\begin{thm}\label{may}  Fix $\theta,\alpha>0$.  Then there exists a constant $C$ depending only on $\theta,\alpha$ such that the following holds.   Suppose that $({\mathcal A}, {\mathcal L}, {\mathscr P}, B, x, \theta)$ obeys Hypothesis 1 at some subset ${\mathcal L}'$ of ${\mathcal L}$.  Write $k := \# {\mathcal L}$, and suppose that $x \geq C$, $B \leq x^\alpha$, and $C \leq k \leq \log^{1/5} x$.  Moreover, assume that the coefficients $a_i, b_i$ of the linear forms $L_i(n) = a_i n + b_i$ in ${\mathcal L}$ obey the size bound $|a_i|, |b_i| \leq x^\alpha$ for all $i=1,\dots,k$.  
Then there exists a smooth function $F: \R^k \to \R$ depending only on $k$ and supported on the simplex $\mathcal{R}_k$, and quantities $I_k,J_k$ depending only on $k$ with
$$ I_k \gg (2k \log k)^{-k}$$
and
\begin{equation}\label{jk}
 J_k \asymp \frac{\log k}{k} I_k
\end{equation}
such that, for $w(n)$ given in terms of $F$ as above, the following assertions hold uniformly for $x^{\theta/10}\le R\le x^{\theta/3}$.
\begin{itemize}
\item We have
\begin{equation}\label{a-big}
 \sum_{n \in {\mathcal A}(x)} w(n) = \(1 + O \pfrac{1}{\log^{1/10} x} \) \frac{B^k}{\varphi(B)^k} {\mathfrak S}({\mathcal L}) \# {\mathcal A}(x) (\log R)^k I_k.
\end{equation}
\item For any linear form $L(n) = a_L n + b_L$ in ${\mathcal L}'$ with $a_L$ coprime to $B$ and $L(n)>R$ on $[x,2x]$, we have
\begin{equation}\label{b-big}
\begin{split}
 \sum_{n \in {\mathcal A}(x)} 1_{\mathscr P}(L(n)) w(n) &= \(1 + O \pfrac{1}{\log^{1/10} x} \) \frac{\phi(|a_L|)}{|a_L|}\frac{B^{k-1}}{\varphi(B)^{k-1}} {\mathfrak S}({\mathcal L}) \# {\mathscr P}_{L,{\mathcal A}}(x) (\log R)^{k+1} J_k\\
&\quad + O\( \frac{B^k}{\varphi(B)^k} {\mathfrak S}({\mathcal L}) \# {\mathcal A}(x)	 (\log R)^{k-1} I_k \).
\end{split}
\end{equation}
\item Let $L(n) = a_0 n + b_0$ be a linear form such that the discriminant
$$ \Delta_L := |a_0| \prod_{j=1}^k |a_0 b_j - a_j b_0|$$
is non-zero (in particular $L$ is not in ${\mathcal L}$).  Then
\begin{equation}\label{c-big}
 \sum_{n \in {\mathcal A}(x)} 1_{\mathscr P \cap [x^{\theta/10},+\infty)}(L(n)) w(n) \ll \frac{\Delta_L}{\varphi(\Delta_L)} \frac{B^k}{\varphi(B)^k} {\mathfrak S}({\mathcal L}) \# {\mathcal A}(x) (\log R)^{k-1} I_k.
\end{equation}
\item We have the crude upper bound
\begin{equation}\label{crude}
 w(n) \ll x^{2\theta/3+o(1)}
\end{equation}
for all $n \in \Z$.
\end{itemize}
Here all implied constants depend only on $\theta,\alpha$ and the implied constants in the bounds of Hypothesis 1.
\end{thm}

\begin{proof} 
The first estimate \eqref{a-big} is given by \cite[Proposition 9.1]{maynard-dense}, \eqref{b-big} follows from \cite[Proposition 9.2]{maynard-dense} in the case $(a_L,B)=1$, \eqref{c-big} is given by \cite[Propositon 9.4]{maynard-dense} (taking $\xi := \theta/10$ and $D := 1$), 
and the final statement \eqref{crude} is given by part (iii) of \cite[Lemma 8.5]{maynard-dense}. The bounds for $J_k$ and $I_k$ are given by \cite[Lemma 8.6]{maynard-dense}.
\end{proof}

We remark that the estimate \eqref{c-big} is only needed here to establish the estimate \eqref{wcp} which is not, strictly speaking, necessary for the results of this paper, but will be useful in a subsequent work \cite{FMT} based on this paper.
\section{Verification of sieve estimates}\label{sec:verification}

We can now prove Theorem \ref{weight}.  Let $x,y,r,h_1,\dots,h_r$ be as in that theorem.
  
We set
\begin{align*}
{\mathcal A} &:= \Z,  \\
\theta &:= 1/3,\\
k &:= r, \\
R &:= (x/4)^{\theta/3},
\end{align*}
and let $B = x^{o(1)}$ be the quantity from Lemma \ref{landau-page}.

We define the function $w: \PP \times \Z \to\R^+$ by setting
$$ w(p,n) := 1_{[-y,y]}(n) w_{k, {\mathcal L}_p, B, R}(n)$$
for $p \in \PP$ and $n \in \Z$, where ${\mathcal L}_p$ is the (ordered) collection of linear forms $n \mapsto n+h_i p$ for $i=1,\dots,r$, and $w_{k, {\mathcal L}_p, B, R}$ was defined in \eqref{wdef}.  Note that the admissibility of the $r$-tuple $(h_1,\dots,h_r)$ implies the admissibility of the linear forms $n \mapsto n+h_i p$.

A key point is that many of the key components of $w_{k, {\mathcal L}_p, B, R}$ are essentially uniform in $p$.  Indeed, for any prime $s$, the polynomial $\prod_{i=1}^k (n+h_i p)$ is divisible by $s$ only at the residue classes $-h_i p \mod s$.  From this we see that
$$ \omega_{{\mathcal L}_{p}}(s) =\#\{h_i\pmod{s}\} \hbox{ whenever } s \neq p.$$
In particular, $ \omega_{{\mathcal L}_{p}}(s)$ is independent of $p$ as long as $s$ is distinct from $p$, so
\begin{align}
{\mathfrak S}( {\mathcal L}_p )& = \(1 + O\(\frac{k}{x}\)\) {\mathfrak S},\label{slp}\\
{\mathfrak S}_{BW}\( {\mathcal L}_p \)& = \(1 + O\(\frac{k}{x}\)\) {\mathfrak S}_{BW},\nonumber
\end{align}
for some ${\mathfrak S}$, ${\mathfrak S}_{BW}$ independent of $p$, with the error terms uniform in $p$. Moreover, if $s\nmid WB$ then $s>2k^2$, so all the $h_i$ are distinct $\mod{s}$ (since the $h_i$ are less than $2k^2$). Therefore, if $s\nmid pWB$ we have $\omega_{\mathcal{L}_p}(s)=k$ and
$$ \{j_{s,1}({\mathcal L}_{p}),\dots,j_{s,\omega(s)}(\mathcal{L}_{p})\} = \{1,\dots,k\}. $$
Since all $p\in\PP$ are at least $x/2>R$, we have $s\ne p$ whenever $s\le R$. From this we see that $\mathcal{D}_k(\mathcal{L}_p)\cap \{(d_1,\dots,d_k):\prod_{i=1}^k d_i\le R\}$ is independent of $p$, and so we have
\[
\lambda_{(d_1,\dots,d_k)}({\mathcal L}_p) =\frac{\mathfrak{S}(\mathcal{L}_p)}{\mathfrak{S}}\lambda_{(d_1,\dots,d_k)}= \(1 + O\(\frac{k}{x}\)\) \lambda_{(d_1,\dots,d_k)},
\]
for some $\lambda_{(d_1,\dots,d_k)}$ independent of $p$, and where the error term is independent of $d_1,\dots,d_k$.

It is clear that $w$ is non-negative and supported on $\PP \times [-y,y]$, and from \eqref{crude} we have \eqref{w-triv}.  We set
\begin{equation}\label{tau-def}
 \tau := 2\frac{B^k}{\varphi(B)^k} {\mathfrak S} (\log R)^k (\log x)^k I_k 
\end{equation}
and
\begin{equation}\label{u-def}
 u := \frac{\varphi(B)}{B} \frac{\log R}{\log x} \frac{k J_k}{2I_k}.
\end{equation}
Since $B$ is either $1$ or prime, we have 
$$
 \frac{\varphi(B)}{B} \asymp 1,
$$
and from definition of $R$ we also have
\begin{equation}\label{rx}
 \frac{\log R}{\log x} \asymp 1.
\end{equation}
From \eqref{jk} we thus obtain \eqref{u-bound}.  From \cite[Lemma 8.1(i)]{maynard-dense} we have
$$ {\mathfrak S} \geq x^{-o(1)},$$
and from \cite[Lemma 8.6]{maynard-dense} we have
$$ I_k = x^{o(1)},$$
and so we have the lower bound \eqref{alpha-crude}.  (In fact, we also have a matching upper bound $\tau \leq x^{o(1)}$, but we will not need this.)

It remains to verify the estimates \eqref{wap} and \eqref{wbp}.  We begin with \eqref{wap}.  Let $p$ be an element of $\PP$.  We shift the $n$ variable by $3y$ and rewrite
$$ \sum_{n \in \Z} w(p,n) = \sum_{n \in {\mathcal A}(2y)} w_{k, {\mathcal L}_{p} - 3y, B, R}(n) + O( x^{1-c+o(1)} )$$
where ${\mathcal L}_p - 3y$ denotes the set of linear forms $n \mapsto n+h_ip - 3y$ for $i=1,\dots,k$.  (The $x^{1-c+o(1)}$ error arises from \eqref{w-triv} and roundoff effects if $y$ is not an integer.)  This set of linear forms remains admissible, and
$$
{\mathfrak S}({\mathcal L}_p - 3y) = {\mathfrak S}({\mathcal L}_p) = \(1 + O\(\frac{k}{x}\)\) {\mathfrak S}.$$
The claim \eqref{wap} now follows from \eqref{tau-def} and the first conclusion \eqref{a-big} of Theorem \ref{may} (with $x$ replaced by $2y$, ${\mathcal L}'=\emptyset$, and ${\mathcal L} = {\mathcal L}_p - 3y$), using Lemma \ref{landau-page} to obtain Hypothesis 1.

Now we prove \eqref{wbp}.  Fix $q \in \QQ$ and $i\in\{1,\dots,k\}$.  We introduce the set $\tilde {\mathcal L}_{q,i}$ of linear forms $\tilde L_{q,i,1},\dots, \tilde L_{q,i,k}$, where
$$ \tilde L_{q,i,i}(n) := n$$
and
$$ \tilde L_{q,i,j}(n) := q + (h_j - h_i) n \qquad (1\le j\le k, j\ne i)$$
We claim that this set of linear forms is admissible.   Indeed, for any prime $s\ne q$, the solutions of
$$
n \prod_{j\ne i} ( q + (h_j - h_i) n) \equiv 0 \pmod{s}
$$
are $n\equiv 0$ and $n\equiv -q(h_j-h_i)^{-1}\pmod{s}$ for
$h_j\not\equiv h_i\pmod s$, the number of which is equal to $\#\{h_j\pmod{s}\}$.
Thus,
\begin{align*}
 {\mathfrak S}( \tilde {\mathcal L}_{q,i} ) &= \(1 + O\(\frac{k}{x}\)\) {\mathfrak S},\\
{\mathfrak S}_{BW}( \tilde {\mathcal L}_{q,i} ) &= \(1 + O\(\frac{k}{x}\)\) {\mathfrak S}_{BW},
\end{align*}
as before. Again, for $s\nmid WB$ we have that the $h_i$ are distinct $\pmod{s}$, and so if $s<R$ and $s\nmid WB$ we have $\omega_{\tilde {\mathcal L}_{q,i}}(s) =k$ and
$$ \{j_{s,1}(\tilde{{\mathcal L}}_{q,i}),\dots,j_{s,\omega(s)}(\tilde{\mathcal{L}}_{q,i})\} =\{1,\dots,k\}.$$
 In particular, $\mathcal{D}_k(\tilde{\mathcal{L}}_{q,i})\cap \{(d_1,\dots,d_k):\prod_{i=1}^kd_i\le R\}$ is independent of $q,i$ and so
\begin{align*}
\lambda_{(d_1,\dots,d_k)}(\tilde {\mathcal L}_{q,i}) &= \(1 + O\(\frac{k}{x}\)\) \lambda_{(d_1,\dots,d_k)},
\end{align*}
where again the $O(\frac{k}{x})$ error is independent of $d_1,\dots,d_k$.  From this, since $q-h_i p$ takes values in $[-y,y]$, we have that
$$ w_{k, \tilde {\mathcal L}_{q,i}, B, R}(p) = \(1 + O\(\frac{k}{x}\)\) 
w_{k, {\mathcal L}_p, B, R}(q-h_i p)$$
whenever $p \in \PP$ (note that the $d_i$ summation variable implicit on both sides of this equation is necessarily equal to $1$). Thus, recalling that $\mathcal{P}=\mathscr{P}\cap(x/2,x]$, we can write the left-hand side of \eqref{wbp} as
$$ \(1 + O\(\frac{k}{x}\)\) \sum_{n \in {\mathcal A}(x/2)} 1_{\mathscr P}(\tilde L_{q,i,i}(n)) w_{k, \tilde {\mathcal L}_{q,i}, B, R}(n).$$
Applying the second conclusion \eqref{b-big} of Theorem \ref{may} (with $x$ replaced by $x/2$, ${\mathcal L}'= \{\tilde L_{q,i,i}\}$, and ${\mathcal L} = \tilde {\mathcal L}_{q,i}$) and using Lemma \ref{landau-page} to obtain Hypothesis 1, this expression becomes 
\begin{align*}
& \(1 + O \pfrac{1}{\log_2^{10} x} \) \frac{B^{k-1}}{\varphi(B)^{k-1}} {\mathfrak S} \# {\mathscr P}_{\tilde L_{q,i,i},{\mathcal A}}(x/2) (\log R)^{k+1} J_k\\
&\quad + O\( \frac{B^k}{\varphi(B)^k} {\mathfrak S} \# {\mathcal A}(x/2)	 (\log R)^{k-1} I_k \).
\end{align*}
Clearly $\# {\mathcal A}(x/2) = O(x)$, and from the prime number theorem one has
$$ \# {\mathscr P}_{\tilde L_{q,i,i},{\mathcal A}}(x/2) = \(1 + O \pfrac{1}{\log_2^{10} x} \) \frac{x}{2 \log x}.$$
for any fixed $C>0$.  Using \eqref{tau-def}, \eqref{u-def}, we can thus write the left-hand side of \eqref{wbp} as
$$
\(1 + O \pfrac{1}{\log_2^{10} x} \) \frac{u}{k} \tau \frac{x}{2\log^k x} + O\( \frac{1}{\log R} \tau \frac{x}{\log^k x} \).$$
From \eqref{r-bound}, \eqref{u-bound}, the second error term may be absorbed into the first, and \eqref{wbp} follows.

Finally, we prove \eqref{wcp}.  Fix $h=O(y/x)$ not equal to any of the $h_i$, and fix $p \in \PP$.  By the prime number theorem, it suffices to show that
$$ \sum_{q \in \QQ} w( p, q - hp ) \ll \frac{1}{\log^{10}_2 x} \tau \frac{y}{\log^k x}.$$
By construction, the left-hand side is the same as
$$ \sum_{x-hp < n \leq y-hp} 1_{\mathscr{P}}(n+hp) w_{k, {\mathcal L}_p, B, R}(n)$$
which we can shift as
$$ \sum_{n \in {\mathcal A}(y-x)} 1_{\mathscr{P} \cap [x^{\theta/10},+\infty)}(n-y+2x) w_{k, {\mathcal L}_p-y+2x-hp, B, R}(n) + O( x^{1-c+o(1)} )$$
where again the $O(x^{1-c+o(1)})$ error is a generous upper bound for roundoff errors.  This error is acceptable and may be discarded.  Applying \eqref{c-big}, we may then bound the main term by
$$
\ll
\frac{\Delta}{\varphi(\Delta)} \frac{B^k}{\varphi(B)^k} {\mathfrak
  S}({\mathcal L_p} - y+2x-hp) y (\log R)^{k-1} I_k =
\frac{\Delta}{\varphi(\Delta)} \frac{B^k}{\varphi(B)^k} {\mathfrak
  S}({\mathcal L_p}) y (\log R)^{k-1} I_k
$$
where
$$ \Delta := \prod_{j=1}^k |hp-h_ip|.$$
Applying \eqref{slp}, \eqref{tau-def}, we may simplify the above upper bound as
$$
\ll
\frac{\Delta}{\varphi(\Delta)} \frac{y}{(\log R) (\log x)^k} \tau.
$$
Now $h-h_i = O(y/x) = O( \log x )$ for each $i$, hence $\Delta
\le (O(x\log x))^k$, and it follows from 
\eqref{xlog}, \eqref{rx} and \eqref{r-bound}  that
$$
\frac{\Delta}{\varphi(\Delta)} \ll \log_2 \Delta \ll \log_2 x \ll
\frac{\log R}{\log^{10}_2 x}.
$$

This concludes the proof of Theorem \ref{weight}, and hence Theorem \ref{mainthm}.

\end{document}